\documentclass[a4paper,twoside,fleqn,12pt]{article}

\usepackage{epsfig}
\usepackage[shortlabels]{enumitem}
\usepackage{amsmath}
\usepackage{amssymb}
\usepackage{amsthm}

\newcommand{\commentaar}[1]{{}}
\newcommand{\comment}[1]{{}}

\newcommand{\field}[1]{\mathbb{#1}}
\newcommand{\inprod}[2]{\langle #1 , #2 \rangle}

\newcommand{\poisson}[2]{\{#1,#2\}}

\newcommand{\system}[1]{\left\{\begin{aligned} #1 \end{aligned}\right.}
\newcommand{\norm}[1]{\lVert #1 \rVert}

\newcommand{\R}{\field{R}}

\newcommand{\SE}{\text{SE}}
\newcommand{\SO}{\text{SO}}

\newcommand{\becomes}{\mathrel{\mathop:}=}

\newcommand{\grad}{\textrm{grad}}
\newcommand{\mat}{\textrm{mat}}
\newcommand{\id}{\textrm{id}}

\newcommand{\bigo}{\textrm{O}}
\newcommand{\hpm}{\hphantom{-}}
\newcommand{\so}{\text{so}}
\newcommand{\cF}{{\mathcal F}}
\newcommand{\Rstar}{\R^{\ast}}

\newcommand{\config}{{ \mathcal{X} }}
\newcommand{\rot}{{\text{R}}}

\newtheorem{theorem}{Theorem}[section]
\newtheorem{lemma}[theorem]{Lemma}
\newtheorem{proposition}[theorem]{Proposition}
\newtheorem{corollary}[theorem]{Corollary}
\newtheorem{definition}[theorem]{Definition}
\newtheorem{Remark}[theorem]{Remark}

\newenvironment{remark}{\begin{Remark} \begin{rm}}{\end{rm} \end{Remark}}

\title{Critical points of the integral map of the charged $3$-body problem}
\author{I. Hoveijn, H. Waalkens, M. Zaman\\
\small{Johann Bernoulli Institute for Mathematics and Computer Science}\\
\small{University of Groningen, The Netherlands}}

\begin{document}

\commentaar{
- Introduction
  - n-body problem
  - symmetry group
  - potential(s)
  - integral manifolds
  - critical points
  - central configurations
  - outline of method
- Main results
- Potential(s) (see chap 2)
- Critical points in $n$-body problem (see chap 2)
- Central configurations
  - collinear central configurations
    - space of collinear central configurations
    - reduced space of collinear central configurations
    - polynomial equation for collinear central configurations on reduced space
    - special points on the discriminant set
    - number of zeros of the polynomial
    - action of the permutation group
    - number of collinear central configurations
    - remarks
  - non-collinear central configurations
- Critical points in charged $3$-body problem
  - Morse types?
- Discussion
}

\commentaar{
edit as of May 25, IH
- the signs of the variables x, y and z are gone, in the main text we only consider x>0, y>0, z<0
  other cases (order of bodies) in the appendix
- uniform notation for parameter family f, except for the appendix but explicitly mentioned
- m is the list of masses, scalar m is replaced with \mu
- the \mu are replaced with \beta
- \R_{>0} replaced with \R^{\ast}
- two references added: lms2015 and apc2008
- first incomplete draft of 'outlook'
}


\maketitle

\section*{Abstract}
This is the first in a series of three papers where we study the integral manifolds of the charged three-body problem. The integral manifolds are the fibers of the map of integrals. Their topological type may change at critical values of the map of integrals. Due to the non-compactness of the integral manifolds one has to take into account besides `ordinary' critical points also critical points at infinity. In the present paper we concentrate on `ordinary' critical points and in particular elucidate their connection to central configurations. In a second paper we will study critical points at infinity. The implications for the Hill regions, i.e. the projections of the integral manifolds to configuration space, are the subject of a third paper. 

\textbf{Keywords:} Charged three-body problem, integral manifold, critical point, relative equilibrium, central configuration

\section{Introduction}\label{sec:intro}
In this paper we consider a system of $N$ point masses with position vectors $q_i\in \R^3$ and masses $m_i>0$ whose dynamics is described by Newton's second law of motion $m_i\ddot{q}_i=F_i(q_1,\ldots,q_N)$, $i=1,\ldots,N$, with forces 
$$
F_i(q_1,\ldots,q_N) = \sum_{i \ne j} F_{ij}(q_i,q_j) ,
$$
where $F_{ij}(q_i,q_j)$ is the force on the point mass $i$ at position $q_i$ due to the point mass $j$ at position $q_j$ which is assumed to be of the form 
\begin{equation} \label{eq:charged_force}
F_{ij}(q_i,q_j) = -\gamma_{i j} \frac{1}{\Vert q_i-q_j\Vert^3} (q_i-q_j). 
\end{equation}
Here
$\gamma_{ij}=\gamma_{ji}\in \R$ for $i \ne j$.
Examples are gravitational forces where $\gamma_{ij} = m_i m_j>0$, electrostatic forces where $\gamma_{ij}= - Q_i Q_j$ with $Q_i$ being the charge of the point mass $i$ which can be positive or negative, or combinations of gravitational and electrostatic forces. In the gravitational case the forces $F_{ij}$ are always attractive. This case is well studied even though also here there are many open questions. For $N=2$ and different signs of the charges $Q_1$ and $Q_2$, the electrostatic force is also attractive and repulsive otherwise. For $N\ge 3$, the electrostatic forces will always involve repulsion, namely repulsion only if all charges have the same sign, or coexisting repulsive and attractive forces when the charges do not all have the same sign. So it is to be expected that for $N\ge3$, systems interacting via electrostatic forces behave qualitatively differently from systems interacting via gravitational forces. 
We will refer to the point masses as bodies (as is common in celestial mechanics where the point masses represent planets or stars) or particles (as is more appropriate in the case of electrostatic forces where we think of the point masses representing electrons or atomic nuclei). We will mainly be concerned with the case of three bodies or particles, i.e. $N=3$. 
The forces are singular when $q_i=q_j$ for some $i\ne j$. We define the \emph{collision set}
\begin{equation*}
\Delta_c = \{ (q_1,\ldots,q_N) \in \R^{3N} \;|\; q_i = q_j \text{ for some } i < j \}.
\end{equation*}
The configuration space is then given by $\config = \R^{3N}\setminus \Delta_c$ and we will often write $q$ for $(q_1,\ldots,q_N)$ in $\config$.

The forces $F_i$ of the form \eqref{eq:charged_force} are conservative. In fact, for $V:\config \to \R$ with 
$$
V(q) = \sum_{i<j} -\frac{\gamma_{ij}}{\Vert q_i-q_j \Vert} ,
$$
we have $F_i=-\partial_{q_i}V$.

We wish to consider the system above as a Hamiltonian dynamical system. Therefore we introduce the \emph{phase space} 
$M = T^*\config \simeq \config \times \R^{3N} $, the cotangent bundle of $\config$, which is endowed with the symplectic two-form
$\omega = \sum_i dp_i \wedge dq_i$ where
$q$ and $p$ are the canonical coordinates on $T^*\config$.
The equations of motion can now be derived from the function $H:M\to \R$ given by $H(q, p) = \sum_i \frac{1}{2m_i} \norm{p_i}^2 + V(q)$ from $(\dot{q}, \dot{p}) = X_H(q,p)$ where $X_H$ is the vector field defined by $dH(\cdot)\, = \omega(X_H, \cdot )$. We obtain the familiar Hamilton equations
\begin{equation} \label{eq:HamiltonEquations}
\begin{split}
\dot{q}_i &= \hpm \frac{\partial H}{\partial p_i} = \frac{1}{m_i} p_i \\
\dot{p}_i &= - \frac{\partial H}{\partial q_i}= - \frac{\partial V}{\partial q_i}
\end{split}
\end{equation}
which are equivalent to Newton's equations above.

The potential $V$ is invariant under the special Euclidean group $\SE(3)=\SO(3) \rtimes \R^3$ which is the semi-direct product of the group of rotations $\SO(3)$ and the group of translations $\R^3$, i.e. 
$V(\rot\, (q_1+a),\ldots,\rot \,(q_N+a))=V(q_1,\ldots,q_N)$ for all $\rot \in \SO(3)$ and $a\in \R^3$. Many of the results presented in this paper will only depend on this property. An $\SE(3)$-invariant potential that only consists of pairwise interactions can be considered as a function of the inter particle distances $r_{ij} = \norm{q_i-q_j}$, $1 \le i < j \le N$ only (see \cite{weyl1997} for $\SE(3)$-invariants). We use this property to define an $N$-body system to which then many of the presented results will apply, and we define the charged $N$-body problem as a special case for which then more can be said than in the case of a general $N$-body system.

\begin{definition}\label{def:NBodySystem}
For $M = T^*\config$ where $\config = \R^{3N}-\Delta_c$ with $\Delta_c$ being a closed subset of $\R^{3N}$ (the collision set), we call the Hamiltonian system $(M,\omega,H)$ an $N$-body system with masses $m_i>0$ of the $i$-th body, $i=1,\ldots,N$, if $H$ is of the form kinetic plus potential energy
$$
H = T + V,
$$
where the kinetic energy is $T(p) = \sum_i \frac{1}{2m_i} p_i^2$ and $V$ is a function of the $\SE(3)$-invariants $r_{ij} = \norm{q_i-q_j}, \quad 1\le i < j \le N,$ only. We will refer to the system that has potential 
$$
V(q) = \sum_{i<j} -\frac{\gamma_{ij}}{r_{ij}(q)}
$$ 
with $\gamma_{ij}\in\R$ as the \emph{charged $N$-body system}. 
\end{definition}

Recall that in case of electrostatic interactions the coefficients $\gamma_{ij}$ are given by $-Q_i Q_j$ with the $Q_i$ being the charges of the particles. 
We note that even for $N=3$, the map $\R^N \to \R^{N(N-1)/2}$, $(Q_1,\ldots,Q_N) \mapsto (Q_1Q_2,Q_1Q_3,\ldots Q_1Q_N, Q_2Q_3, \ldots Q_2Q_N, \ldots, Q_{N-1} Q_N)$ is not bijective and not even surjective. But we will ignore this fact. Our definition of a charged $N$-body system also includes the gravitational $N$-body system. 
Many of the results will not depend on the special form of the potential that defines a charged $N$-body system but only on the $\SE(3)$ symmetry that more generally defines an $N$-body system. These results then also apply, e.g., to an $N$-atomic molecules where the potential might be a complicated function obtained from the Born-Oppenheimer approximation. 

Finding solutions to the equations of motion is in general very difficult. Exploiting the symmetries, the problem of two bodies, $N=2$, can be reduced to one degree of freedom. The gravitational Kepler problem is the most famous example (see, e.g., \cite{arnold1980} for a detailed account of the solution). For $N=3$, the situation is however already considerably more complicated. Even for the well-studied case of the gravitational three-body problem only few solutions are known. The most prominent ones are the ones found by Euler and Lagrange where the three bodies move along ellipses with the three bodies being at each instant of time on a line or at the vertices of an equilateral triangle, respectively. These solutions belong to the class of so called \emph{homographic} solutions where the bodies move individually along conic sections and together form configurations that remain invariant under similarity transformations. 
An important building block for such solutions are \emph{central configurations} which are special points in configuration space that have the property that for each body $i$, the force $F_i$ is proportional to the distance vector of the $i$-th body to the common centre of mass where the proportionality constant is the same for all bodies \cite{lms2015}.
\begin{definition}\label{def:ccfg}
A \emph{central configuration} is a point $q=(q_1,\ldots,q_N)$ in configuration space for which there is a \emph{multiplier} $\lambda \in \R$ such that
\begin{equation*}
\frac{\partial}{\partial q_i}V = \lambda m_i (q_i-Q) \;\;\text{for all }\;\; i \in \{1,\ldots,N\}\,.
\end{equation*}
Here $Q=\frac{1}{m}\sum_i m_i q_i $ is the center of mass.
\end{definition}
Only fairly recently another class of solutions given by so called \emph{choreographies} like the figure-eight orbit have been found. They were first found numerically by C. Moore \cite{moore1993} and later proved to exist using variational techniques by A. Chenciner and R. Montgomery \cite{cm2000}. 

Instead of studying individual solutions we restrict ourselves to the study of more global properties of $N$-body systems. By definition $N$-body systems have many symmetries. By Noether's theorem (and its generalizations in terms of momentum maps) the symmetries entail the existence of integrals in addition to the Hamiltonian function $H$, i.e. functions $M\to\R$ whose values are constant along solution curves of Hamilton's equations \eqref{eq:HamiltonEquations}. The solution curves are hence constrained to the joint level sets of these integrals. What we will be interested in is the geometry of these integral manifolds. For the gravitational three-body problem, the topology of the integral manifolds have been studied in great detail \cite{mmw1998}. Using the homogeneity of the gravitational potential it can be shown that they depend on a single scalar parameter for which there are nine critical values at which the topology (apart from one exception) changes. The situation is complicated by the non-compactness of the integral manifolds. This requires one to not only study `ordinary' critical points of the map of integrals but also critical points at infinity (see the work of Albouy \cite{alb1993} where this notion has been made precise). In a series of three papers we carry out a program that provides the first steps towards a similar study as that in \cite{mmw1998} to the charged three-body problem. In the present first paper we will study `ordinary' critical points of the map of integrals. In the second paper we will study critical points at infinity. In the third paper we study the Hill regions which are the projections of the integral manifolds to configuration space $\config$.

The main results of the present paper concerns the number of critical points in a charged three-body system.
\begin{theorem}\label{the:main}
Consider a charged three-body system as in definition~\ref{def:NBodySystem}. Depending on the values of the masses $m_i$ and the values of the parameters $\gamma_{ij}$ in the potential there 
are (up to similarity transformations) zero, one, two or three critical points associated to collinear central configurations. There is at most one critical point associated to non-collinear central configurations.
\end{theorem}

This theorem follows from corollary~\ref{cor:numberccc} and theorem~\ref{the:nccc}. Central configurations play a subtle role in the study of critical points of the map of integrals. In fact there are two types of critical points of the map of integrals of which one is given by relative equilibria, i.e. equilibria of the system reduced by its symmetries. We will show that for every central configuration (of a charged three-body system) with a positive multiplier, there is a relative equilibrium in the sense that the relative equilibrium (a point in phase space) projects to a central configuration (a point in configuration space), see corollary~\ref{cor:pccfgtocp}. We note that for a charged $N$-body system, the situation is more complicated than for a gravitational $N$-body system. In the gravitational case every relative equilibrium projects to a planar central configuration. We will show that this remains to be the case also in the charged case when all $\gamma_{ij}$ have the same sign (proposition~\ref{pro:eqsign}). In the case of mixed signs this need not be true as we demonstrate in the third paper~\cite{hwz2018c}.

The present paper is organized as follows. In section~\ref{sec:symmetries_and_integralmap} we use the symmetries of $N$-body systems to define the map of integrals. Critical points of $N$-body systems with a general $\SE(3)$-invariant potential are the subject of section~\ref{sec:cripts}. In section~\ref{sec:ccfg} we consider central configurations in a system with a Newton-Coulomb potential but with arbitrary coefficients. Afterwards we single out the central configurations related to critical points. Some remarks and an outlook are given in section~\ref{sec:outlook}. Some details are discussed in  appendices. 

\commentaar{hierboven ook nog iets over kritieke punten op oneindig en al het andere dat we onder 'remarks and outlook' kwijt zouden willen}

\section{Symmetries and the integral map}\label{sec:symmetries_and_integralmap}
In the introduction we based our definition of an $N$-body system as a Hamiltonian system on the invariance of the potential under the action of the special Euclidean group $\SE(3)$. In this section we make the notion of symmetries of Hamiltonian system more formal. We will distinguish between symplectic group actions such as translations and rotations and non-symplectic group actions. The symplectic group actions will lead to the definition of the integral map whose study is the main subject of this paper. The homogeneity of the potential for charged $N$-body systems gives rise to a non-symplectic group action which facilitates conclusions on the integral manifolds that go further than in the case of an $N$-body system which does not have this symmetry.\commentaar{mention $\SE(3)$?}

\subsection{Symplectic actions}\label{sec:SymplecticActions}
Here we use a general construction for group actions initially defined on the configuration space $ \R^{3N}$ only. Here we ignore the fact that we might have to exclude a collision set $\Delta_c$ from the configuration space. The reason is that the collision set and its complement are invariant under the group actions of interest. Starting from group actions on $ \R^{3N}$ it is then easy to define their restrictions to $\config = \R^{3N} \backslash \Delta_c$. First we lift this action to a (symplectic) action on the phase space $T^*\R^{3N}$. Then we construct a Hamiltonian function for the action of one-parameter subgroups. Throughout we assume that we consider an $N$-body problem as in definition~\ref{def:NBodySystem}. Our exposition will be very short. For a more complete and detailed account, see for example Abraham \& Marsden \cite{am1987} or Arnol'd \cite{arnold1980}.

Our starting point is the action of a (smooth) group $G$ on the configuration space $\R^{3N}$. For a fixed element $g$ of the group, the group action is a diffeomorphism $\Phi_g$ on $\R^{3N}$. As such the cotangent lift of $\Phi_g$ is given by $\Phi_g^{\sharp} : T^*\R^{3N} \to T^*\R^{3N} : (q, p) \mapsto (\Phi_g(q), d_{\phi_g(q)}(\Phi_g^{-1})^*(p))$. This map is symplectic on $T^*\R^{3N}$ by construction, i.e. $\Phi_g^{\sharp*} \omega = \omega$. Now let $\Phi_t$ be the action on $\R^{3N}$ of a one-parameter subgroup and let $X$ be the vector field on $\R^{3N}$ generating the flow $\Phi_t$. Then $F : T^*\R^{3N} \to \R : (q, p) \mapsto \sum_i p_i \xi_i(q)$ is the Hamiltonian of the flow $\Phi^{\sharp}_t$ on $T^*\R^{3N}$, where the $\xi_i$ are the components of $X$. This means that $F$ generates the flow $\Phi^{\sharp}_t$ in the sense that the vector field 
$X_F=\left. \frac{d}{dt}\Phi^{\sharp}_t\right|_{t=0}$ on $T^*\R^{3N}$ satisfies $\omega(X_F,\cdot)=d F$.

The Hamiltonian system $(M,\omega,H)$ is said to be invariant under the symplectic action of the group $G$ if for each fixed group element $g$, the pullback of $H$ by the cotangent lift $\Phi_g^{\sharp}$ is equal to $H$, i.e. $\Phi_g^{\sharp*} H=H$. In terms of the function $F$ generating the vector field $X_F$ associated with the action of a one-parameter subgroup this means $\poisson{F}{H}=0$ where $\poisson{\cdot}{\cdot}$ is the Poisson bracket of $F$ and $H$ which is defined via the symplectic two-form $\omega$ as $\poisson{F}{H} = \omega(X_F,X_H)$ which gives
\begin{equation*}
\poisson{F}{H} = \sum_i \frac{\partial F}{\partial q_i} \frac{\partial H}{\partial p_i} - \frac{\partial F}{\partial p_i} \frac{\partial H}{\partial q_i} \,.
\end{equation*}

\paragraph{Translation symmetry.} The potential $V$ of an $N$-body system is by definition invariant under the translation group $\R^3$ acting on the configuration space $\R^{3N}$ as $(q_1,\ldots,q_N)\mapsto (q_1+a,\ldots,q_N+a) $ for each $a\in \R^3$. The derivative of this action for a fixed group element is the identity. Therefore the cotangent lift of this action to $T^*\R^{3N}$ is given by
\begin{equation*}
\Phi_a^{\sharp} : M \to M : (q_1,\ldots,q_N,p_1,\ldots,p_N) \mapsto (q_1+a,\ldots,q_N+a,p_1,\ldots,p_N).
\end{equation*}
The translation group is a three-parameter group generated by the translations over $t\,e_j$ for $t \in \R$ and $e_j$ the $j$th basis vector of $\R^3$. For each of these, the vector field generating the flow is constant, $X_j = \sum_i \frac{\partial}{\partial q_{i,j}}$. The integrals corresponding to this action are the components of total momentum $P(q,p) = \sum_i p_i$\cite{am1987,arnold1980}. The Hamiltonian $H$ of the $N$-body system is invariant under $\Phi^{\sharp}_a$ as is easily checked: $H(\Phi^{\sharp}_a(q,p)) = H(q,p)$ for all $a \in \R^3$ and in each point $(q,p) \in M$.  
The commutativity of $\R^3$ implies $\poisson{P_i}{P_j} = 0$ for $i, j\in \{1,2,3\}$.

Since the map $P$ restricted to the levels of $H$ does not have critical values, we may without loss of generality choose any value for $P$. The total momentum determines the time derivative of the center of mass $Q(q,p) = \frac{1}{m} \sum_i m_i q_i$, where $m = \sum_i m_i$ is the total mass, according to 
\begin{equation*}
\frac{d}{dt}Q = \poisson{Q}{H} = \frac{1}{m} \sum_i m_i \frac{p_i}{m_i} = \frac{1}{m}P.
\end{equation*}
The implication is that choosing the value zero for $P$ fixes the center of mass. Indeed, the time derivative of $Q$ is zero. Thus the components of the center of mass are also integrals if total momentum equals zero. Since the function $Q$ does not have critical points we may without loss of generality set the value of $Q$ equal to zero as well.

\paragraph{Rotation symmetry.} By definition the potential function $V$ of an $N$-body system is also invariant under the action of the group $\SO(3)$ on configuration space. The action of $\SO(3)$ on $\R^{3N}$ is 'diagonal'. Let $\rot \in \SO(3)$ then each $q_i$ is mapped to $\rot \,q_i$. The derivative of this map is $\rot$ and since $(\rot^{-1})^T = \rot$ the cotangent lift of the action on $\R^{3N}$ to an action on $T^*\R^{3N}$ is given by
\begin{equation*}
\Phi_\rot^{\sharp} : M \to M : (q_1,\ldots,q_N,p_1,\ldots,p_N) \mapsto (\rot q_1,\ldots, \rot q_N,\rot p_1,\ldots,\rot p_N).
\end{equation*}
The group $\SO(3)$ is not commutative, but it is still a three-parameter group generated by rotations around the coordinate axes in $\R^3$. Each of these forms a 
one-parameter subgroup and the vector field generating the rotations are linear. For example, the rotations $\rot_x$ around the $x$-axis are generated by 
\begin{equation*}
X = \sum_i \bigg( q_{i,3} \frac{\partial}{\partial q_{i,2}} - q_{i,2} \frac{\partial}{\partial q_{i,3}} \bigg).
\end{equation*}
So the Hamiltonian function of the flow is $\sum_i (p_{i,2} q_{i,3} - p_{i,3} q_{i,2})$ which is the first component of total angular momentum: the integrals of the $\SO(3)$-action are the components of angular momentum $L(q,p) = \sum_i q_i \times p_i$. Again the Hamiltonian $H$ of the $N$-body system is invariant under $\Phi^{\sharp}_\rot$ for each $\rot \in\SO(3)$  and thus each component of $L$ satisfy $ \poisson{L_j}{H}=0$, $j \in \{1,2,3\}$. However, since $\SO(3)$ is not commutative, the integrals $L_j$ do not Poisson commute. Instead they generate a Lie subalgebra isomorphic to $\so(3)$ in the Lie algebra of smooth functions on $M$ with the Poisson bracket.

\subsection{Non-symplectic actions}
For a charged $N$-body system, the Hamilton function $H$ is according to definition~\ref{def:NBodySystem} of the form $H(q,p) = T(p) + V(q)$ where the kinetic energy $T$ is a homogeneous function of degree 2 in $p$ and the potential $V$ is a homogeneous of degree -1 in $q$. Due to this special property it is useful to also consider a non-symplectic group action called \emph{dilation symmetry}.
The dilation group is given by the multiplicative group $\Rstar = \R \setminus \{0\}$ for which one can define an action on the extended phase space 
$M \times \R$ as 
\begin{equation*}
\Phi : \Rstar \times M \times \R \to M \times \R : (\mu, q, p, t) \mapsto (\mu^{\alpha} q, \mu^{\beta} p, \mu^{\gamma} t).
\end{equation*}
If the exponents $\alpha$, $\beta$ and $\gamma$ satisfy $\alpha = -2\beta$ and $\gamma = -3\beta$ the equations of motion are invariant. However, neither the Hamiltonian $H$ nor the symplectic form $\omega$ are invariant in this case. They transform as $H \mapsto \mu^{2\beta} H$ and $\omega \mapsto \mu^{-\beta} \omega$.

\subsection{Integrals and relative equilibria}\label{sec:relequi}
Let $(M, \omega, H)$ be a Hamiltonian system and suppose that $\Phi_s: M \to M$ is a one-parameter group of symplectic diffeomorphisms such that $H$ is invariant. Furthermore let $F$ be the Hamiltonian function associated with the vector field $X_F$ that generates $\Phi_s$. Then we have
\begin{equation*}
0 = \frac{d}{ds} H(\Phi_s) = dH(X_F) = \omega(X_H, X_F) = \poisson{H}{F}. 
\end{equation*}
This means that $H$ is preserved by the flow of $F$. By the skew symmetry of the Poisson bracket, this also means that $F$ is preserved under the flow of $H$. 
In fact,  $\poisson{H}{F}.=0$ means that the symplectic vector fields $X_F$ and $X_H$ commute, and equivalently the corresponding flows commute. 
A function $F$ that is preserved by the flow of $H$ has several names: \emph{conserved quantity}, \emph{conserved function}, \emph{constant of motion} or \emph{integral}. 
The existence of an integral has several important implications for the dynamics. The integral can, e.g., lead to a special type of solutions of the Hamiltonian system associated to the corresponding symmetry. 
This happens when the vector fields are linearly dependent at some point $(q, p) \in M$, i.e. $X_H(q, p) = \lambda X_F(q, p)$, for some
 $\lambda \in \R$. Due to the $\Phi_s$-invariance this equation holds for the $\Phi_s$-orbit through $(q, p)$. The orbit of $\Phi_s$ through $(q, p)$ thus coincides with the solution curve of $H$ through $(q, p)$. Such orbits are called \emph{relative equilibria} because on the reduced system, where $\Phi_s$-orbits are points, the reduced Hamiltonian has an equilibrium point.

From a slightly different point of view we also have the following. If the vector fields $X_F$ and $X_H$ are linearly dependent, the gradients of $F$ and $H$ are also linearly dependent and vice versa. Thus we get the equation $\grad H(q, p) = \lambda\, \grad F(q, p)$, which can be interpreted as the equation for critical points of $H$ restricted to the levels of $F$. The conclusion is that relative equilibria are in one to one correspondence with critical points of $H$ restricted to the levels of an integral.

\subsection{The integral map of an $N$-body system} 

As mentioned above
the existence of an integral $F$ restricts the dynamics of $H$ to level sets of the function $F$. This simplifies the analysis of the (dynamical) problem. In fact the Liouville theorem states that if for a Hamiltonian system there are sufficiently many integrals then the system is \emph{integrable} \cite{arnold1980}. More precisely if the system is $2n$-dimensional, there must be $n$ independent, Poisson commuting integrals. Considered as a Hamiltonian system, the $N$-body system is $6N$-dimensional. In section~\ref{sec:SymplecticActions} we have seen that that the invariance of a $N$-body system under translations and rotations leads to 9 integrals, the three components of total momentum $P$ and the center of mass $Q$ which we all assume to have the value zero 
and the three components of the angular momentum $L$. Together with the Hamiltonian $H$ this gives 10 integrals. All the 10 integrals commute with $H$ with respect to the Poisson bracket. Even though the integrals do not commute mutually they render the $1$-body and $2$-body systems integrable. For $N\ge3$, this however does not give enough integrals to guarantee integrability. Nevertheless the integrals can be used to reduce the dimension of the system. In fact our focus will be on the \emph{integral manifolds}, i.e. the fibers of the \emph{integral map}. We use the 10 integrals of an $N$-body system to define the \emph{integral map} $\cF$ as 
\begin{equation*}\label{eq:integralmap}
\cF : M \to \R^{10} : (q, p) \mapsto \big( H(q, p), L(q, p), P(q, p), Q(q, p) \big).
\end{equation*}
Rather than in the dynamics of the $N$-body system we will be interested in the fibers of the integral map.
For most values of $\cF$, the fibers of $\cF$ are manifolds, which we call \emph{integral manifolds}. A value $\phi$ of $\cF$ such that in any open neighbourhood of $\phi$ fibers exist that are not diffeomorphic, is called a \emph{bifurcation value} of $\cF$. This means that upon crossing a bifurcation value the topology of the integral manifolds may change. The values of $\cF$ at critical points of $\cF$ in $M$ are example of such bifurcation values. Due to the non-compactness of the fibers bifurcations can occur not only at such critical values but also due to critical points at infinity \cite{alb1993}. Our aim is to find the bifurcation values in the spirit of \cite{mmw1998} for the charged $N$-body system with $N = 3$ where in the present paper we will concentrate on bifurcation values due to `ordinary' critical points and in a forthcoming paper we will study critical points at infinity \cite{hwz2018b}.

\section{Critical points of the integral map}\label{sec:cripts}
Consider the $N$-body system as defined in definition~\ref{def:NBodySystem}. Let us find critical points of the integral map $\cF$ by considering the rank of $D\cF$ restricted to points in phase space where $P(q, p) = 0$ \emph{and} $Q(q, p) = 0$. We first formulate and prove a statement on the rank of $D\cF$ and after that give an interpretation of the result and draw some conclusions, see also \cite{lms2015}.


\begin{proposition}\label{pro:rankdf_a}
The derivative $D\cF$ of $\cF$ at $(q,p)$ does not have full rank if and only if at least one of the following three conditions holds
\begin{enumerate}

\item the $q_i$ and $p_i$, $i=1,\ldots,N$,  are multiples of single vector $e \in \R^3$,

\item  $(q,p)$ is an equilibrium point, or

\item $(q,p)$ is a relative equilibrium with respect to the $SO(3)$ action.

\end{enumerate}
\end{proposition}

We prove this proposition using the following two lemmas.

\begin{lemma}\label{lem:gradvmat}
For $i=1,\ldots,N$, the $i$-th component of the gradient of $V$ is a linear combination of the position vectors $q_1,\ldots,q_N$. Specifically we have
$$
\frac{\partial}{\partial q_i}V = \sum_j \alpha_{ij} q_j,
$$
where, using $r_{ij} = \norm{q_i - q_j}$,
$$
\alpha_{ii} = \sum_{i<k} \frac{\partial V}{\partial r_{ik}} \frac{1}{r_{ik}} + \sum_{k<i} \frac{\partial V}{\partial r_{ki}} \frac{1}{r_{ki}} ,
$$
and for $i\ne j$,
$$
\alpha_{ij} = \left\{
\begin{array}{cc}
-\frac{\partial V}{\partial r_{ij}} \frac{1}{r_{ij}} & \text{ if } i<j\\
-\frac{\partial V}{\partial r_{ji}} \frac{1}{r_{ji}} & \text{ if } j<i\
\end{array}
\right. .
$$
\end{lemma}

\begin{proof}
By direct computation we get
\begin{eqnarray*}
\frac{\partial}{\partial q_i}V &=& \sum_{j<k} \frac{\partial V}{\partial r_{jk}} \frac{\partial r_{jk}}{\partial q_i}\\
&=& \sum_{j<k} \frac{\partial V}{\partial r_{jk}} \big( \frac{q_j-q_k}{r_{jk}}\delta_{ij} + \frac{q_k-q_j}{r_{jk}}\delta_{ik} \big) \\
&=& \sum_{i<k} \frac{\partial V}{\partial r_{ik}} \frac{q_i-q_k}{r_{ik}} + \sum_{j<i} \frac{\partial V}{\partial r_{ji}} \frac{q_i-q_j}{r_{ji}} \\
&=& \big( \sum_{i<k} \frac{\partial V}{\partial r_{ik}} \frac{1}{r_{ik}} + \sum_{k< i} \frac{\partial V}{\partial r_{ki}} \frac{1}{r_{ki}} \big) q_i - 
 \sum_{i<j} \frac{\partial V}{\partial r_{ij}} \frac{1}{r_{ij}} q_j - \sum_{i>j} \frac{\partial V}{\partial r_{ji}} \frac{1}{r_{ji}} q_j,
\end{eqnarray*}
where the sums in the first and second line are double sums, $\delta_{ij}$ ($\delta_{jk}$) in the second line denotes the Kronecker symbol so that the sums in the third and fourth line become single sums, and in the fourth line we relabeled the summation index. The result follows.
\end{proof}

\begin{lemma}\label{lemma:TotalForceVanishes}
The total force on an $N$-body system is zero, i.e.
$$
\sum_{i} \frac{\partial}{\partial q_i}V =0.
$$
\end{lemma}

\begin{proof}
We sum the third line of the equation in the proof of the previous lemma over $i$ to get

\begin{eqnarray*}
\sum_{i} \frac{\partial}{\partial q_i}V &=& \sum_{i<k} \frac{\partial V}{\partial r_{ik}} \frac{q_i-q_k}{r_{ik}} + \sum_{j<i} \frac{\partial V}{\partial r_{ji}} \frac{q_i-q_j}{r_{ji}} \\
&=& \sum_{i<k} \frac{\partial V}{\partial r_{ik}} \frac{q_i-q_k}{r_{ik}} + \sum_{i<k} \frac{\partial V}{\partial r_{ik}} \frac{q_k-q_i}{r_{ik}} , \\
\end{eqnarray*}
where the sums after the first and second equality symbol are double sums, and we relabeled the summation indices in the second line. The result follows. We note that the result follows even faster from the conservation of the total momentum $P$, namely
$$0 = \dot{P} = \sum_i \dot{p_i} = \sum_i \poisson{p_i}{H} = \sum_{i} \frac{\partial}{\partial q_i}V.$$

\end{proof}

Let us now proceed with the proof of proposition~\ref{pro:rankdf_a}.

\begin{proof} (Proposition~\ref{pro:rankdf_a})
The transpose matrix of the derivative of the map $D\cF$ at $(q,p)$ is given by
\begin{equation}\label{eq:df}
\big(D\cF(q, p)\big)^T = \begin{pmatrix}
\frac{\partial}{\partial q_1}V & -A_{p_1} & 0 & \frac{m_1}{m} I \\
\vdots & \vdots & \vdots & \vdots \\
\frac{\partial}{\partial q_N}V & -A_{p_N} & 0 & \frac{m_N}{m} I \\[0.75ex]
\frac{p_1}{m_1} & \hpm A_{q_1} & I & 0 \\
\vdots & \vdots & \vdots & \vdots \\
\frac{p_N}{m_N} & \hpm A_{q_N} & I & 0
\end{pmatrix},
\end{equation}
where $I$ is the identity matrix on $\R^3$ and $A_x$ is the skew symmetric matrix associated with the cross product on $\R^3$, namely $A_x(y) = x \times y$ so that
\begin{equation*}
A_x = \begin{pmatrix} 0 & -x_3 & x_2 \\ x_3 & 0 & -x_1 \\ -x_2 & x_1 & 0 \end{pmatrix}.
\end{equation*}
$D\cF(q,p)$ does not have full rank if a non-zero vector $\lambda = (\lambda_0, \lambda_1, \lambda_2, \lambda_3) \in \R \times \R^3 \times \R^3 \times \R^3$ exists such that $\big(D\cF(q, p)\big)^T \lambda = 0$. The last six columns of $\big(D\cF(q,p)\big)^T$ are clearly independent. 
Hence, for $D\cF(q,p)$ not to have full rank, at least one of $\lambda_0$ and $\lambda_1$ must be non-zero. We now distinguish two main cases, namely $\lambda_0 = 0$ and $\lambda_0 \neq 0$.

\begin{enumerate}[itemsep=0pt, topsep=0pt]
\item Suppose that $\lambda_0 = 0$.  Then we get $2N$ equations from $\big(D\cF(q, p)\big)^T \lambda = 0$, namely
\begin{equation*}
\system{
0 &= \frac{m_i}{m} \lambda_3 - p_i \times \lambda_1,\\
0 &= \lambda_2 + q_i \times \lambda_1
}
\end{equation*}
for $i \in \{1,\ldots,N\}$.  As $\lambda_1 = 0$ would imply $\lambda_2 = 0$ and $\lambda_3 = 0$ which together with $\lambda_0=0$ does not give a critical point it follows that $\lambda_1 \neq 0$. Adding the first $N$ equations and adding the second $N$ equations with weights $m_i/m$ we get
\begin{equation*}
\system{
0 &= \lambda_3 \frac{1}{m} \sum_i m_i - \big( \sum_i p_i \big) \times \lambda_1 = \lambda_3, \\
0 &= \lambda_2 \frac{1}{m} \sum_i m_i + \big( \frac{1}{m} \sum_i m_i q_i \big) \times \lambda_1 = \lambda_2 ,
}
\end{equation*}
where we used $0 = P= \sum_i p_i$ and $0 = Q = \frac{1}{m} \sum_i m_i q_i$. Therefore $\lambda_2 = 0$ and $\lambda_3 = 0$. We are left with the equations $q_i \times \lambda_1 = 0$ and $p_i \times \lambda_1 = 0$ which imply that the $q_i$ and the $p_i$, $i=1,\ldots,N$, are multiples of the same vector $\lambda_1$. If, conversely, we assume that the $q_i$ and $p_i$ are multiples of the same vector $e$, then taking $\lambda = (0, e, 0, 0)$ shows that rank $D\cF$ is less than 10.
\item Now suppose that $\lambda_0 \neq 0$. Without restriction we may assume $\lambda_0 = 1$. From $\big(D\cF(q, p) \big)^T \lambda = 0$ we again get $2N$ equations
\begin{equation*}
\system{
0 &= \frac{\partial}{\partial q_i}V + \frac{m_i}{m} \lambda_3 - p_i \times \lambda_1 , \\
0 &= \frac{p_i}{m_i} + \lambda_2 + q_i \times \lambda_1 ,
}
\end{equation*}
for $i \in \{1,\ldots,N\}$. Adding the equations in the same manner as in case~1 we get
\begin{equation*}
\system{
0 &= \sum_i \frac{\partial}{\partial q_i}V + \lambda_3 \frac{1}{m} \sum_i m_i - \big( \sum_i p_i \big) \times \lambda_1 = \lambda_3 ,\\
0 &= \frac{1}{m} \sum_i p_i + \lambda_2 \frac{1}{m} \sum_i m_i + \big( \frac{1}{m} \sum_i m_i q_i \big) \times \lambda_1 = \lambda_2 .
}
\end{equation*}
Here we again use $0 = P = \sum_i p_i$ and $0 = Q = \frac{1}{m} \sum_i m_i q_i$, and $ \sum_i \frac{\partial}{\partial q_i}V=0$ as we showed in lemma~\ref{lemma:TotalForceVanishes}. So again we find $\lambda_2 = 0$ and $\lambda_3 = 0$ and we are left with the equations
\begin{equation}\label{eq:critpt_equil_case}
\system{
\frac{\partial}{\partial q_i}V &= p_i \times \lambda_1,\\
p_i &= m_i \lambda_1 \times q_i,
}
\end{equation}
For $\lambda_1=0$, we see that $(q,p)$ must be an equilibrium point. 
The case $\lambda_1\ne0$, on the other hand, gives critical points of the Hamiltonian restricted to levels of angular momentum. Indeed  we have $0 = \big(D\cF(q, p)\big)^T \lambda = \nabla_{(q,p)} H + \lambda_{11} \nabla_{(q,p)} L_1 + \lambda_{12} \nabla_{(q,p)} L_2 + \lambda_{13} \nabla_{(q,p)} L_3$, where the $L_j$ are the components of angular momentum. With the characterization of relative equilibria in section~\ref{sec:relequi} we see that these points are in indeed relative equilibria with respect to the action of $\SO(3)$. 

\end{enumerate}

\end{proof}

Let us now comment on proposition \ref{pro:rankdf_a}.

From case~1 in the proof of proposition \ref{pro:rankdf_a} we see that critical points of this type do not depend on the specific Hamiltonian of the $N$-body system. These points are in fact critical points of angular momentum. The reason is that in the vector $\lambda$ only component $\lambda_1$ is non-zero and $(D\cF(q, p))^T \lambda$ in equation~\eqref{eq:df} thus only selects $L$. The $N$ bodies are moving on a single line in $\R^3$ through the center of mass. In that case angular momentum is zero. So zero is a critical value. Needless to say that level zero of angular momentum also contains other motions of the $N$ bodies than just motion on a line.

%

The questions arises whether an $N$-body system admits equilibrium points and relative equilibria.
In the gravitational case we have the following.

\begin{proposition}\label{pro:statpts}
There are no equilibrium points in the gravitational $N$-body problem.
\end{proposition}
\begin{proof}
As $V$ is a homogeneous function of degree $-1$ of the position vectors we have by Euler's homogeneous function theorem
\begin{equation*}
\sum_{i=1}^{N} q_i \frac{\partial V}{\partial q_i} = -V \,.
\end{equation*}
As  $V$ is a strictly negative function the left hand side of the equation cannot vanish which implies that $V$ has no critical points. Hence this implies that the gravitational $N$-body problem has no equilibrium  points.
\end{proof}

\begin{remark}
In the proof of proposition~\ref{pro:statpts} we used the fact that the potential for the gravitational $N$-body problem is homogeneous and  everywhere negative. For the charged $N$-body problem the potential is still homogeneous. However the potential does in general not have a constant sign. So the existence of equilibrium points is not excluded by the argument employed in  proposition~\ref{pro:statpts}. The proof shows that if there are equilibrium  solutions for the charged case then they project to the zero-level set of the potential. 
\end{remark}

Concerning relative equilibria we have the following proposition.


\begin{proposition}\label{pro:rankdf}
Planar central configurations with positive multiplier give rise to relative equilibria. More specifically, if there exists a vector $e\in \R^3$ such that the $q_i$ and $p_i$ satisfy 
\begin{enumerate}[itemsep=0pt, topsep=0pt]
\item $\inprod{q_i}{e} = 0$,
\item $\frac{\partial}{\partial q_i}V = \inprod{e}{e} m_i q_i$ and 
\item $p_i = m_i e \times q_i$.\commentaar{projection on configuration space is planar central configuration}
\end{enumerate}
for $i \in \{1,\ldots,N\}$, then the point $(q,p)$ is a relative equilibrium. Conversely, if a relative equilibrium projects onto a central configuration then the central configuration is planar and has a positive multiplier.
\end{proposition}

\begin{proof} 
Recall that case~2 of the proof of proposition~\ref{pro:rankdf_a}  with $\lambda_1\ne0$ gives a relative equilibrium. 
Now note that 
equation~\eqref{eq:critpt_equil_case}  that corresponds to this case is equivalent to 
\begin{equation}\label{eq:critpt}
\system{
\frac{\partial}{\partial q_i}V &= m_i (\lambda_1 \times q_i) \times \lambda_1 = m_i \big( \langle \lambda_1, \lambda_1 \rangle q_i - \langle \lambda_1, q_i \rangle \lambda_1 \big),\\
p_i &= m_i \lambda_1 \times q_i.
}
\end{equation}
This means that the vectors $\frac{\partial}{\partial q_i}V$ are in a plane perpendicular to $\lambda_1$. We have $\frac{\partial}{\partial q_i}V = \sum_j \alpha_{ij}(q)\, q_j$ with the $\alpha_{ij}$ as defined in lemma~\ref{lem:gradvmat}. Since the $\frac{\partial}{\partial q_i}V$ are perpendicular to $\lambda_1$ we must have
\begin{equation}\label{eq:kernel}
\sum_j \alpha_{ij}(q) \inprod{q_j}{\lambda_1} = 0.
\end{equation}
Let $(\alpha_{ij})$ be the matrix with entries $\alpha_{ij}$. Then $(\alpha_{ij})$ is symmetric and by lemma~\ref{lemma:TotalForceVanishes} the row sums of $(\alpha_{ij})$ are zero. 
%
The fact that the row sums of $(\alpha_{ij})$ are zero implies that $(1,1,\ldots,1)$ is an element of $\ker (\alpha_{ij})$, which in turn implies that $\inprod{q_j}{\lambda_1} = c$ independent of $j$ is a solution of equation \eqref{eq:kernel}. But then
\begin{equation*}
0 = \langle \sum_i m_i q_i, \lambda_1 \rangle = \sum_i m_i \langle q_i ,\lambda_1 \rangle = c \sum_i m_i
\end{equation*}
implies $c = 0$ because the masses are positive. Thus we are left with the following equations, no longer equivalent to the system in \eqref{eq:critpt},
\begin{equation*}
\system{
\frac{\partial}{\partial q_i}V &= \langle \lambda_1, \lambda_1 \rangle m_i q_i,\\
p_i &= m_i \lambda_1 \times q_i.
}
\end{equation*}
This means that a critical point $(q, p)$ of this type projects onto a planar central configuration $q$ with a positive multiplier.

%
%

Now suppose that a critical point $(q, p)$, i.e. a solution of equation \eqref{eq:critpt}, projects onto a central configuration $q$. 
Then a $\mu$ exists such that $\frac{\partial}{\partial q_i}V = \mu m_i q_i$. Since the $\frac{\partial}{\partial q_i}V$ are perpendicular 
to $\lambda_1$ we must have $\inprod{\lambda_1}{q_i}=0$. So $q$ must be a planar central configuration. Again using equation~\eqref{eq:critpt}
we see that the multiplier must be positive.

\end{proof}


The proposition does not necessarily provide all relative equilibria.  Depending on the potential $V$ there may be more. 
In our proof this fact is related to the kernel of the matrix $A$ associated to the gradient of the potential, see lemma \ref{lem:gradvmat}. This kernel is at least spanned by the vector $(1,1,\ldots,1)$. The gravitational potential is an example where this vector is the only spanning vector of $\ker(A)$. Thus in the gravitational case the proposition gives all relative equilibria.

Propositions~\ref{pro:rankdf_a} and \ref{pro:rankdf} relate central configurations, see definition~\ref{def:ccfg}, and critical points of the integral map given by relative equilibria. The relation  is subtle, not only because central configurations reside in configuration space and critical points in phase space. 
To make the relation more clear we give a list of corollaries.

Irrespective of the potential $V$ (apart from being 
$\SE(3)$ invariant) we have the following corollaries.
\begin{corollary}\label{cor:pccfgtocp}
For every planar central configuration $q$ with a positive multiplier there exists a critical point $(q, p)$.
\end{corollary}
\begin{corollary}\label{cor:cptopccfg}
If a relative equilibrium $(q, p)$ projects onto a central configuration $q$, the latter must be a planar central configuration with a positive multiplier.
\end{corollary}
If we do know more about the potential $V$ then we have the following.
\begin{corollary}\label{cor:onedimker}
If $(1,1,\ldots,1)$ is the only spanning vector of $\ker(A)$, see lemma~\ref{lem:gradvmat}, then every relative equilibrium $(q, p)$ projects onto a planar central configuration with a positive multiplier.
\end{corollary}
\begin{corollary}\label{cor:ndimker}
If $\dim(\ker(A)) > 1$, then relative equilibria do not necessarily project onto central configurations.
\end{corollary}
Determining the kernel of $A$ is in general not easy. We do however have another criterion in the next proposition.
\begin{proposition}\label{pro:eqsign}
Suppose the derivatives $\partial_{r_{ij}} V$ of the potential $V$ of an $N$-body system have equal signs for all $1\le i < j \le N$. Then every relative equilibrium projects onto a planar central configuration.
\end{proposition}
The proof of this proposition will be given below. In view of corollary~\ref{cor:cptopccfg} the multiplier of this central configuration must be positive. Proposition~\ref{pro:eqsign} implicitly states that under the conditions mentioned, the kernel of $A$ is one-dimensional. Needless to say that knowing this kernel is not our ultimate goal. An example to which proposition~\ref{pro:eqsign} applies is the gravitational $N$-body system for which a proof can already be found in \cite{wintner1941}. For this system, there is a relative equilibrium for every planar central configuration and every relative equilibrium projects onto a central configuration. In particular all central configurations are planar. This need not be so for a charged $N$-body system. If in the case of electrostatic interactions all charges have the same sign then there is no relative equilibrium (all forces are repulsive). If the charges have different signs then there are indeed examples of relative equilibria that do not project onto a central configuration, see \cite{hwz2018c}.\commentaar{vooruit ref naar artikel 3}

\begin{proof} (Proposition~\ref{pro:eqsign})
Let $(q, p)$ be a solution of equation \eqref{eq:critpt}. Suppose $\partial_{r_{ij}}V>0$ for all $1\le i < j \le N$. Considering the equations in \eqref{eq:critpt} in case~2 of the proof of proposition~\ref{pro:rankdf} we need to show that $\partial_{q_i}V$ being orthogonal to a $\lambda\in \R^3$, $\lambda\ne0$, for $i=1,\ldots,N$, implies that the $q_i$ are orthogonal to $\lambda$ for all $i=1,\ldots,N$. Then we infer from equation \eqref{eq:critpt} that $q$ is a central configuration with a positive multiplier.

Now we show by contradiction that $\partial_{q_i}V$ is orthogonal to a $\lambda\in \R^3$, $\lambda\ne0$, for $i=1,\ldots,N$. To this end let $\Pi$ be the plane orthogonal to $\lambda$. As we require $Q=0$, $\Pi$ contains the origin. Without restriction we may assume that $\lambda=(0,0,1)$ (otherwise rotate the coordinate system). Suppose one or more $q_i$ are not contained in the plane $\Pi$. From this $q_i$ choose one that is furthest away from $\Pi$. Without restriction we can assume that this $q_i$ lies above $\Pi$, i.e. has a positive $z$-component (otherwise rotate the coordinate system). Note that the maximum $z$ can be attained by more than one $q_i$. However no more than $N-1$ many can attain the maximal $z$ because otherwise all $q_i$ would have to be in a plane and because of $Q=0$ this plane would be $\Pi$. Without restriction we can take $q_1$ to be the one with maximal $z$ and $q_2$ to have a smaller $z$-component (otherwise relabel). Then the $z$-component of the force on particle 1 is 
$$
F_{1z} = -\partial_{z_1}V= -\sum_{i=2}^N \frac{\partial V}{\partial r_{1i}} \frac{z_1-z_i}{r_{1i}} \le \frac{\partial V}{\partial r_{12}} \frac{z_1-z_2}{r_{12}} < 0
$$
which contradicts $\partial_{q_i}V$ being orthogonal to $\lambda=(0,0,1)$ for all $i=1,\ldots,N$.
The same kind of arguments apply to the case $\partial_{r_{ij}}V<0$ for $1\le i < j \le N$.
\end{proof}

\section{Central configurations}\label{sec:ccfg}
In the previous section we have seen that for an $N$-body system, planar central configurations give rise to critical points. In this respect three-body systems are exceptional because here all central configurations are planar. More generally central configurations turn up in $N$-body systems when we look for solutions whose configuration remains similar during the motion, so called \emph{homographic} solutions. This means that during motion only size and orientation of the configuration change. Imposing these conditions on the configuration $(q_1,\ldots,q_N)$, $q_i \in \R^3$ leads to the equation for a central configuration in definition~\ref{def:ccfg}. Central configurations are invariant under the action of the Euclidean group $\SE(3)$, and if the potential is homogeneous like in the case of a charged $N$-body system then it is also invariant under dilations.
\begin{lemma}\label{lem:cccinvrtd}
The equation defining a central configuration in definition~\ref{def:ccfg} is invariant under 
translations and rotations. For a charged $N$-body system, 
it is moreover invariant under dilations. 
\end{lemma}

\begin{proof}
Clearly the defining equation of a central configuration in definition~\ref{def:ccfg} 
is translation invariant. Since the $\partial_{q_i}V$ transform under rotations like the $q_i$ and a rotation of the $q_i$ also induces a rotation of the center of mass $Q$ the defining equation of a central configuration is invariant under rotations. 
The potential of the charged $N$-body system is homogeneous of degree $-1$. If we define in this case the action of dilations as $(q, \lambda) \mapsto (tq, t^{-3}\lambda)$ for $t\in \Rstar$, the defining equation is also dilation invariant.
\end{proof}

In this section we find the central configurations for the charged three-body system, i.e. for a three-body system with potential
\begin{equation}\label{eq:potential}
V(q) = - \frac{\alpha_3}{\norm{q_1-q_2}} - \frac{\alpha_1}{\norm{q_2-q_3}} - \frac{\alpha_2}{\norm{q_3-q_1}},
\end{equation}
where $\alpha_i = \gamma_{jk}$ for $(i,j,k)=(1,2,3)$ and cyclic permutations. Recall from the introduction that in the case of electrostatic interactions the map from the charges $Q_i$ to the coefficients $\gamma_{ij}=-Q_i Q_j$ is not onto. Indeed for three bodies, the map $\R^3 \to \R^3 : x \mapsto y = (-x_2 x_3, -x_1 x_3, -x_1 x_2)$ has image $\{y \in \R^3 \;|\; y_1 y_2 y_3 \leq 0\}$, a collection of four octants. As mentioned in the introduction we will ignore this fact.

It turns out that we have to distinguish between collinear and non-collinear central configurations.

\subsection{Outline of the method}
Our approach differs from that in the literature (see \cite{pssy1996}) in particular with respect to how we treat collinear central configurations. We start with identifying the space of collinear configurations. On this space the translation, rotation, dilation and permutation groups act. Then we first find the reduced space of collinear configurations by reducing with respect to translations, rotations and dilations. In case of three bodies the reduced space is one-dimensional and the condition for central configurations reduces to a single equation. However, due to the fact that the masses are undetermined and the potential contains unknown coefficients we have in fact a parameter family of equations. When we use the potential from equation \eqref{eq:potential} the collinear central configurations are determined by the zeroes of a fifth degree polynomial. In general it will be impossible to find explicit solutions to this equation. But the solutions depend on the parameters and by using failure of the implicit function theorem we are at least able to find \emph{regions} in parameter space with a constant number of solutions.

We will take a more geometric point of view and consider the equation for central configurations on the product of the parameter space and the reduced space. Then the equation defines a hyper surface. Projecting the latter to the parameter space we obtain the aforementioned regions by locating the projection singularities. (As it turns out, the surface itself has no singularities.) The simplest singularities are folds and where a fold is parallel to the projection direction we get more complicated singularities. We carry out this program in some detail for the potential in equation~\eqref{eq:potential}.

In the system of three charged bodies, the parameter space is six-dimensional: there are three masses $m_i$ and three coefficients $\alpha_i$ in the potential, $i \in \{1,2,3\}$. However, the reduced equation for the collinear central configurations is homogeneous in both the $m_i$ and the $\alpha_i$. Thus each set of three can be reduced to the ratios $[m_1 : m_2 : m_3]$ and $[\alpha_1 : \alpha_2 : \alpha_3]$. At this point we prefer to consider the parameter family of equations as a parameter family (parameterized by the $m_i$) of parameter families (parameterized by the $\alpha_i$). This splitting is very useful in identifying the aforementioned regions in the parameter plane with coordinates $(\frac{\alpha_1}{\alpha_3}, \frac{\alpha_2}{\alpha_3})$ (assuming $\alpha_3 \neq 0$). But when we consider permutations of the bodies, the picture becomes slightly more complicated, see appendix~\ref{sec:gactions}.

To determine the regions we consider the fifth degree polynomial mentioned before, as a parameter family for each $m$ (the aforementioned parameter family of parameter families). As it turns out the topology of the regions in the $[\alpha_1 : \alpha_2 : \alpha_3]$-space does not depend on $m$. Thus we are left with the problem of classifying the number of zeroes of a polynomial depending on two parameters. Regions in the parameter plane with a constant number of zeroes are bounded by curves where the polynomial has double zeroes. From a geometrical point of view, these are fold curves. Therefore we also expect isolated cusp points which we find indeed.

To find the non-collinear central configurations we also first apply reduction with respect to translational and rotational symmetry on configuration space. We will see that the condition to have a central configuration is equivalent to having a critical point of the potential restricted to levels of the moment of inertia. The latter condition can easily be stated in the reduced configuration space. Using the $\SE(3)$ invariant mutual body distances $r_{ij}$ as coordinates the computation of non-collinear central configurations is straightforward.

\subsection{Collinear central configurations}\label{sec:cccs}
In this section we consider collinear central configurations for charged three-body systems and we carry out the program sketched in the previous section. The main result for three bodies is the following.
\begin{theorem}\label{the:ccc}
Consider a three-body system with potential as in equation \eqref{eq:potential}. If we take the following values for the masses and coefficients of the potential $m_1 \becomes \mu$, $m_2 \becomes \mu$ and $m_3 \becomes 1$ and furthermore $\alpha_3 \becomes 1$, then there are 13 different regions in the $(\alpha_1, \alpha_2)$ parameter plane with a constant number of collinear central configurations. Up to permutations this number can be $0$, $1$, $2$ or $3$.
\end{theorem}

The reduced space $\R \cup \{\infty\}$ of collinear configurations has a natural splitting into three intervals $I_1 \becomes (\infty, -1)$, $I_2 \becomes (-1,0)$  and $I_3 \becomes (0,\infty)$, see section~\ref{sec:rscc}. In 
figure~\ref{fig:gregions} together with table~\ref{tab:gregions} the number of real zeros of the reduced equation on each of the intervals $I_i$ on the reduced space are indicated. Each of these intervals corresponds to a certain order of the bodies on a line.

\begin{figure}[htbp]
\setlength{\unitlength}{1mm}
%
\begin{picture}(40,73)(0,10)
\put(0,   0){\includegraphics[scale=0.75]{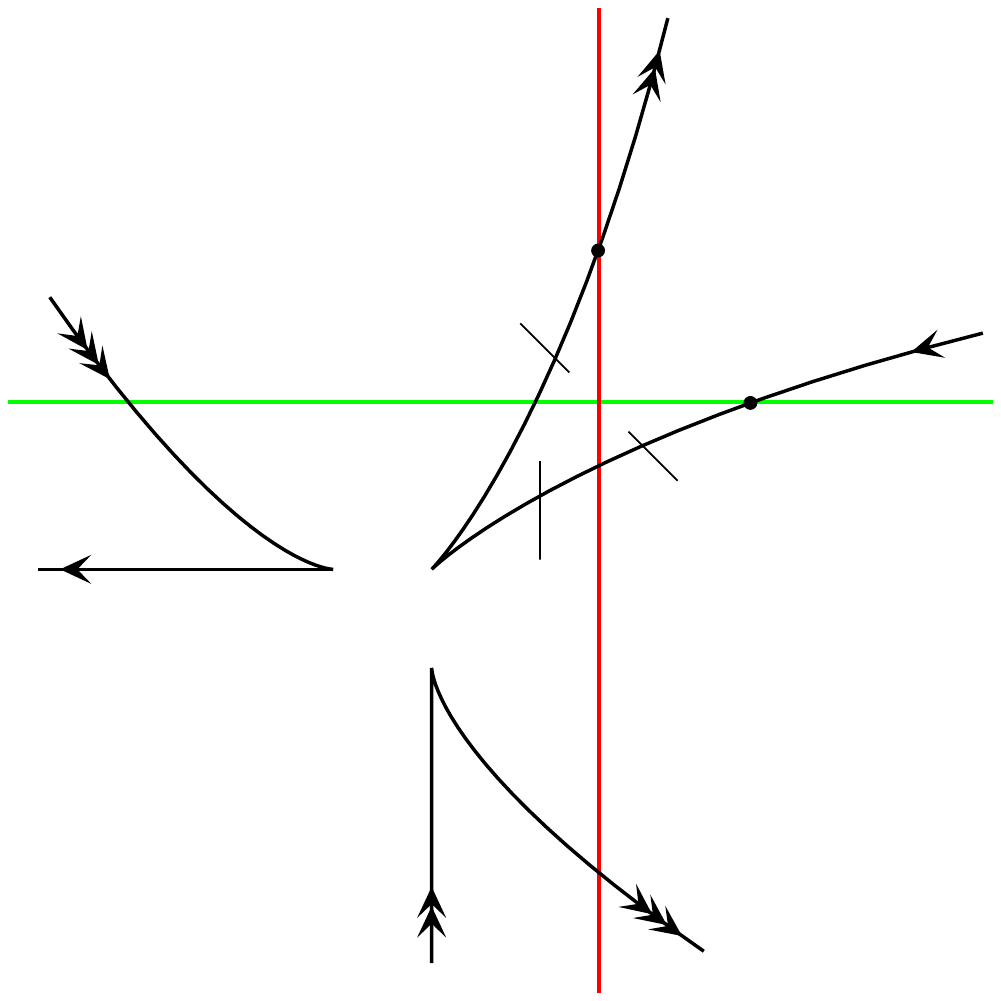}}
\put(79, 49){$\alpha_1$}
\put(47, 80){$\alpha_2$}
\put(65, 50){\rotatebox{-45}{\scriptsize{$u=0$}}}
\put(40, 38){\rotatebox{-45}{\scriptsize{$u=1$}}}
\put(44, 72){\rotatebox{-45}{\scriptsize{$u=\infty$}}}
\put(62,  9){\scriptsize{$u=-1$}}
\put(9,  62){\scriptsize{$u=-1$}}
\put(70, 70){\scriptsize{1}}
\put(81, 54){\scriptsize{2}}
\put(70, 30){\scriptsize{3}}
\put(55,  8){\scriptsize{4}}
\put(45,  8){\scriptsize{5}}
\put(25, 30){\scriptsize{6}}
\put(10, 42){\scriptsize{7}}
\put(10, 54){\scriptsize{8}}
\put(25, 70){\scriptsize{9}}
\put(54, 82){\rotatebox{90}{\scriptsize{10}}}
\put(43, 59){\scriptsize{11}}
\put(58, 43){\scriptsize{12}}
\put(47, 37){\scriptsize{13}}
\end{picture}\\
\begin{picture}(0,0)(-90,-17.5)
\put(0,  0){\includegraphics[scale=0.75]{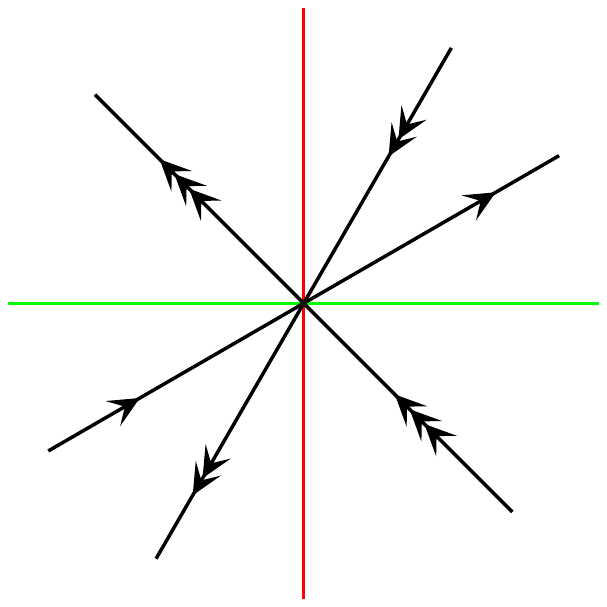}}
\put(50, 27){$\alpha_1$}
\put(27, 55){$\alpha_2$}
\put(32, 47){\scriptsize{10}}
\put(46, 45){\scriptsize{1}}
\put(46, 33){\scriptsize{2}}
\put(46, 26){\scriptsize{3}}
\put(35, 10){\scriptsize{4}}
\put(23, 10){\scriptsize{5}}
\put(13, 15){\scriptsize{6}}
\put(13, 26){\scriptsize{7}}
\put(13, 33){\scriptsize{8}}
\put(23, 47){\scriptsize{9}}
\end{picture}
\caption[Regions in the $(\alpha_1, \alpha_2)$-parameter plane]{\textit{Regions in the $(\alpha_1, \alpha_2)$-parameter plane with different numbers of real zeros of the polynomial $f$ in equation~\eqref{eq:fdecomp} in the intervals $I_i$. The regions are numbered 1 to 13 and in table \ref{tab:gregions} the numbers of real zeros are listed. When compactifying the parameter plane by adding one point at infinity, the branches in the left picture meet at infinity as shown in the right picture where the `origin' represents the point added at infinity. The points $u=-1$, $u=0$ and $u=\infty$ correspond to double collisions.}}\label{fig:gregions}
\end{figure}

\begin{table}[htbp]
\begin{center}
\begin{tabular}{l|c|l|c}
region & zeros     & region & zeros\\\hline
1      & $(0,0,1)$ & 7      & $(3,1,1)$\\
2      & $(2,0,1)$ & 8      & $(2,1,0)$\\
3      & $(1,0,0)$ & 9      & $(0,1,0)$\\
4      & $(1,2,0)$ & 10     & $(0,2,1)$\\
5      & $(1,3,1)$ & 11     & $(0,1,2)$\\
6      & $(1,1,1)$ & 12     & $(1,0,2)$\\
       &           & 13     & $(1,1,3)$
\end{tabular}
\end{center}
\caption[The number of zeros of $f$]{\textit{The number of zeros of $f$ in each region in the intervals $I_1$, $I_2$ and $I_3$.}\label{tab:gregions}}
\end{table}

\subsubsection{The space of collinear configurations}\label{sec:scc}
Let us first identify the space of collinear configurations for an $N$-body system. The point $q=(q_1,\ldots,q_N) $ is a collinear configuration if the differences to the center of mass $q_i -Q$ are multiples of the same vector, i.e. there is 
 a unit vector $v \in \R^3$ such that $q_i-Q = r_i v$, with $r_i \in \R$ and $|r_i| = \norm{q_i -Q }$ for $i \in \{1,\ldots,N\}$. The $r_i$ are not independent as 
 $\sum_i m_i r_i =0$. A collinear configuration is thus determined by $N-1$ many $r_i\in \R$, the center of mass $Q\in \R^3$, and a vector $v $ on the unit sphere $ S^2$. 
Accordingly, the space of collinear configurations is topologically equivalent to $\R^{N-1} \times \R^3 \times S^2$. 
On this space we have four group actions, namely actions of translations, rotations, dilations and permutations. 
Whereas the translations and rotations act on the factors $\R^3$ and $S^2$, respectively, the dilations and permutations act on the first factor $\R^{N-1}$.

\subsubsection{The reduced space of collinear configurations}\label{sec:rscc}
In the following we will reduce the space of collinear configurations by the symmetries. In the case of a charged three-body system, the reduced space will be one-dimensional which will simplify the method of finding collinear central configurations.
\begin{lemma}\label{lem:cccreduc}
For a charged $N$-body system, the space of collinear configurations reduced with respect to translations, rotations and dilations is topologically equivalent to 
$\mathbb{R}P^{N-2}$ 
\end{lemma}

\begin{proof} The space of collinear configurations is $\R^{N-1} \times \R^3 \times S^2$. First we reduce the translations which act on the second factor. The space of orbits of the action of the translation group $\R^3$ is again $\R^3$ and as the reduced space we can take the point where the center of mass $Q$ is at the origin. 
The rotations act on the third factor. Every two points on $S^2$ are $\SO(3)$-equivalent, i.e. there is only one $\SO(3)$-orbit on $S^2$. Therefore the reduced space is a single point which we can choose to be, e.g., the unit vector $e_1$. 
On the remaining factor $\R^{N-1} \cong \{(r_1,\ldots,r_N)\in \R^N \;|\: \sum_i m_i r_i=0\}$
the dilations act as $r_i \mapsto t r_i$ for $t \in \Rstar$. 
The orbits are lines through the origin. Then the reduced space is topologically equivalent to $\mathbb{R}P^{N-2}$.
\end{proof}
\begin{figure}[htbp]
\setlength{\unitlength}{1mm}
\begin{picture}(40,35)(-15,3)
\put(0,    0){\includegraphics[scale=1.0]{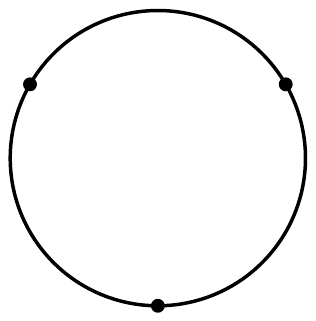}}
\put(35,  27){\scriptsize{$u=0$}}
\put(-4,  27){\scriptsize{$u=-1$}}
\put(16,   2){\scriptsize{$u=\infty$}}
\put(9,   25){\scriptsize{$(1,2)$}}
\put(17,   7){\scriptsize{$(2,3)$}}
\put(25,  25){\scriptsize{$(1,3)$}}
\put(34,  12){$I_3$}
\put(19,  37){$I_2$}
\put(2 ,  10){$I_1$}
\end{picture}
\begin{picture}(50,0)(-15,-8)
\put(0, 0){\includegraphics[scale=1.0]{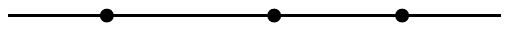}}
\put(19, 5){$r_1$}
\put(49, 5){$r_2$}
\put(36, 5){$r_3$}
\put(43, 12){$x$}
\put(28, 12){$y$}
\end{picture}
\caption[The reduced space of collinear configurations]{\textit{Left: The reduced space $\R \cup \{\infty\}$ of collinear configurations with the intervals $I_1$, $I_2$ and $I_3$. Inside the circle collisions of pairs are indicated at the special values of $u$. Right: A collinear configuration with two relative distances $x$ and $y$.}\label{fig:redccc}}
\end{figure}
In the case of the reduced space of the charged three-body systems we can introduce on the space
$\{(r_1,r_2,r_3)\in \R^3 \;|\; m_1 r_1 + m_2 r_2 + m_3 r_3=0\}$ the differences $x \becomes r_2 - r_3$, $y \becomes r_3 - r_1$ and $z \becomes r_1 - r_2$ to obtain the space $V_3 \becomes \{(x,y,z) \in \mathbb{R}^3 \;|\; x+y+z = 0\}$. Now $V_3$ as a vector space is isomorphic to $\R^2$ and thus the reduced space with respect to dilations is topologically equivalent to $\mathbb{R}P^1$ which is a circle $S^1$. On this circle we can take $u$, defined by $y = ux$, as a coordinate, where $u\in \R \cup \{\infty\}\cong S^1$. 
\begin{definition}\label{def:uintervals}
Let $I_1$, $I_2$ and $I_3$ be three intervals in $\R \cup \{\infty\}$, defined as: $I_1 \becomes (\infty, -1)$, $I_2 \becomes (-1,0)$  and $I_3 \becomes (0,\infty)$.
\end{definition}
Then $u \in I_1$ corresponds to three bodies on a line in the orders $(1, 2, 3)$ or $(3, 2, 1)$, $u \in I_2$ corresponds to the orders $(1, 3, 2)$ or $(2, 3, 1)$ and $u \in I_3$ corresponds to the orders $(3, 1, 2)$ or $(2, 1, 3)$, see 
figure~\ref{fig:redccc}. The boundaries of $I_1$, $I_2$ and $I_3$, i.e. the points $-1$, $0$ and $\infty$, correspond to collisions of bodies 1 and 2, 1 and 3 and bodies 2 and 3, respectively.

\subsubsection{Polynomial equation for collinear central configurations on the reduced space}\label{sec:peqccc}
Now that we know the reduced space of collinear configurations we derive a reduced equation for the collinear central configurations. Using the same coordinates as in the previous section we find a single equation for $u$. In deriving this equation we also encounter the intermediate coordinates $x$, $y$ and $z$. Their signs have a relation with the order of the bodies as is implicit in the discussion of the previous section. Not every choice is allowed since $x + y + z \equiv 0$. Here we take $x > 0$, $y > 0$ and $z < 0$, corresponding to the order $(1, 3, 2)$, see figure \ref{fig:redccc}. Permuting the bodies we get the other possible sign choices, see appendix~\ref{sec:gactions}. On the reduced space we have a single coordinate $u$ defined by $y=ux$ and the identity for $x$, $y$ and $z$ yields $z=-(1+u)x$. Finally we have the following result on collinear central configurations.
\begin{proposition}\label{pro:redeq}
On the reduced space of collinear configurations the equation for central configurations reads $f(u) = 0$, where 
\begin{equation*}
f(u) = \alpha_1 f_1(u) + \alpha_2 f_2(u) + \alpha_3 f_3(u)
\end{equation*}
with 
\begin{equation}\label{eq:fdecomp}
\begin{aligned}
f_1(u) & \becomes m_1 u^2 (1 + u)^2 (m_2 + m_2 u + m_3 u),\\
f_2(u) & \becomes -m_2 (1 + u)^2 (m_1 + m_3 + m_1 u),\\
f_3(u) & \becomes -m_3 u^2 (m_2 - m_1 u).
\end{aligned}
\end{equation}
The polynomial $f$ depends on six parameters: $\alpha_1$, $\alpha_2$, $\alpha_3$, $m_1$, $m_2$ and $m_3$.
\end{proposition}
\begin{proof}
We start the proof with the components of $\frac{\partial V(r)}{\partial r_i} = - m_i \lambda r_i$ for $i \in \{1, 2, 3\}$. We will eventually eliminate $\lambda$, then the resulting equations will be linear in $\alpha_i$. Taking linear combinations of the above equations with coefficients $(m_2, -m_1, 0)$, $(0, m_3, -m_2)$ and $(-m_3,0, m_1)$, we get equations depending on differences $r_i - r_j$ only. Following the reduction steps in the proof of lemma \ref{lem:cccreduc} we switch to variables $x = r_2 - r_3$, $y = r_3 - r_1$ and $z = r_1 - r_2$. Using the choice of signs $x > 0$, $y > 0$ and $z < 0$ we get
\begin{equation}\label{eq:intermed}
\system{
\alpha_1 \frac{m_2+m_3}{x^2} - \alpha_2 \frac{m_2}{y^2} + \alpha_3 \frac{m_3}{z^2} - \lambda m_2 m_3 x & = 0,\\
\alpha_2 \frac{m_1+m_3}{y^2} - \alpha_1 \frac{m_1}{x^2} + \alpha_3 \frac{m_3}{z^2} - \lambda m_1 m_3 y & = 0,\\
\alpha_3 \frac{m_1+m_2}{z^2} + \alpha_1 \frac{m_1}{x^2} + \alpha_2 \frac{m_2}{y^2} + \lambda m_1 m_2 z & = 0.
}
\end{equation}
The sum of these equations with coefficients $(m_1, m_2, -m_3)$ yields $\lambda m_1 m_2 m_3 (x + y + z) = 0$ which is true because $x + y + z$ is identical to zero. However this implies that the equations are linearly dependent (which is not surprising). Eliminating $\lambda$, we will be left with a single equation. Now if $(x, y, z)$ is a solution, then $(tx, ty, tz)$ for $t \in \Rstar$ is also a solution (dilation symmetry). Like in lemma \ref{lem:cccreduc} we set $y = u x$ and $z = -x - u x$. Then we get a single equation for $u$ only which can be written as $f(u) = 0$ as given in the proposition. This polynomial $f$ depends on the parameters $\alpha_i$ and $m_i$.
\end{proof}
\noindent
Note that the polynomial $f$ is homogeneous (degree one) in the $\alpha_i$ and homogeneous (degree two) in the $m_i$. This in turn implies that scaling the parameters $(\alpha_1,\alpha_2,\alpha_3) \mapsto (t\alpha_1,t\alpha_2,t\alpha_3)$ and $(m_1,m_2,m_3) \mapsto (sm_1,sm_2,sm_3)$ with $t, s \in \Rstar$ does not change the solutions $u$ of $f(u) = 0$.

Our ultimate goal is to obtain the number of collinear central configurations for each value of the parameters $\alpha_i$ and $m_i$. In order to do that we have to know how the permutation group of the bodies acts on the reduced space. This will be postponed to appendix~\ref{sec:gactions}. For the moment our aim is to divide the parameters space into regions with a constant number of real solutions of $f(u) = 0$. Since the parameter space is six-dimensional we make some simplifications. First we fix the values of the masses $m_i$ at arbitrary values. Then we scale the parameters $\alpha_i$ according to $(\alpha_1,\alpha_2,\alpha_3) \mapsto (\beta_1,\beta_2,1)=\frac{1}{\alpha_3}(\alpha_1,\alpha_2,\alpha_3)$ assuming $\alpha_3 \neq 0$. Thus we are left with a two-parameter family of polynomials. The discriminant set of $f$ is now a collection of curves in the $(\beta_1,\beta_2)$-parameter plane. On these curves $f$ has double zeros, so by crossing one of these curves the number of real zeros jumps by two. Using a more detailed analysis we will be able to tell how the number of zeros in each interval $I_1$, $I_2$ and $I_3$ changes.
\begin{proposition}\label{pro:parplane}
The intersection of the discriminant set of a general fifth degree polynomial and the two parameter family $f$ consists of three curves: both coordinate axes $\beta_1 = 0$ and $\beta_2 = 0$ and a curve $\Gamma$. These curves separate regions in the $(\beta_1, \beta_2)$-parameter plane. In each region the number of solutions of $f(u) = 0$ in the intervals $I_1$, $I_2$ and $I_3$ is constant. The curve $\Gamma$ is parameterized by $u \in \R \cup \{\infty\}$ and depends on the $m_i$. A parameterization $c : \R \cup \{\infty\} \to \R^2$ of $\Gamma$ is given by
\begin{equation}\label{eq:gammapar}
c(u) = \big( \frac{f_3(u) f_2'(u) - f_2(u) f_3'(u)}{f_2(u) f_1'(u) - f_1(u) f_2'(u)}, \frac{f_1(u) f_3'(u) - f_3(u) f_1'(u)}{f_2(u) f_1'(u) - f_1(u) f_2'(u)} \big),
\end{equation}
where $f_1$, $f_2$ and $f_3$ are defined in equation \eqref{eq:fdecomp}. The curves and regions are shown in figures \ref{fig:gcurve} and \ref{fig:gregions} for the special case $m_1 = m_2$ and $m_3 = 1$.
\end{proposition}
\begin{proof}
We first consider two special cases. Expanding the polynomial $f$ using the definition in equation \eqref{eq:fdecomp} we have:
\begin{align*}
f(u) &= \alpha_1 m_1 (m_2 + m_3) u^5 + \alpha_1 m_1 (3m_2 + 2m_3) u^4\\
     &+ \big( \alpha_1 m_1 (3m_2 + m_3) - \alpha_2 m_1 m_3 + \alpha_3 m_1 m_3 \big) u^3\\
     &+ \big( \alpha_1 m_1 m_2 - \alpha_2 m_2 (3m_1 + m_3) - \alpha_3 m_2 m_3 \big) u^2\\
     &- \alpha_2 m_2 (3m_1 + 2m_3) u - \alpha_2 m_2 (m_1 + m_3).
\end{align*}
From this expression we infer that if $\alpha_1$ tends to zero, two zeros of $f$ tend to infinity and if $\alpha_2$ tends to zero, two zeros of $f$ tend to zero. Therefore on the lines $\alpha_1 = 0$ and $\alpha_2 = 0$, $f$ has double zeros. In general $f$ has double zeros on the discriminant set defined by the equations
\begin{equation*}
\system{f(u) &= 0, \\ f'(u) &= 0.}
\end{equation*}
The equations are linear in the $\alpha_i$. We solve them for $\alpha_1$ and $\alpha_2$ by using Cramer's rule and the decomposition of $f$ defined in equation \eqref{eq:fdecomp}. In the last step we replace the $\alpha_i$ by $\beta_i$ according to $(\alpha_1,\alpha_2,\alpha_3) \mapsto (\beta_1,\beta_2,1)=\frac{1}{\alpha_3}(\alpha_1,\alpha_2,\alpha_3)$ assuming $\alpha_3 \neq 0$. Thus we find the parameterization $c$ of $\Gamma$ in equation \eqref{eq:gammapar}.
\end{proof}
\begin{remark}\strut
\begin{enumerate}[topsep=0ex,partopsep=0px,itemsep=0px]
\item Details about special points on the curve $\Gamma$, like singular points and points at infinity, can be found in appendix~\ref{sec:spoc}.
\item Since we parameterize the curve $\Gamma$ with the zeros $u \in \R$ of $f$ we do not have to worry about pairs of complex conjugate double zeros which in principle are also produced by our method.
\end{enumerate}
\end{remark}

\subsubsection{The number of zeros of the polynomial $f$}\label{sec:nzp}
In the previous paragraph we found regions in the $(\beta_1, \beta_2)$-parameter plane with a constant number of real solutions of $f(u) = 0$. Since these regions are bounded by curves of double zeros of $f$, the number of zeros changes by two upon crossing one of them. As soon as we know the number of real solutions in one region we can find the numbers in the other regions by the following lemma.
\begin{lemma}\label{lem:rules}
On crossing the $\beta_1$-axis the number of zeros in $I_2$ and $I_3$ both increase or decrease by one, and on crossing the $\beta_2$-axis the number of zeros in both $I_1$ and $I_3$ increase or decrease by one. The curve $\Gamma$ is the union of images of the intervals $I_i$ by the parameterization $c$. On crossing the image of $I_i$, the number of real zeros in $I_i$ changes by two.
\end{lemma}
\begin{proof}
On the $\beta_1$-axis the polynomial $f$ has a double zero at $u = 0$, and on the $\beta_2$-axis the $f$ has a double zero at $u = \infty$. On a part of $\Gamma$ which is the image of $I_i$, $f$ has a double zero in $I_i$.
\end{proof}
To use lemma \ref{lem:rules} we determine the number of real solutions in one region.
\begin{lemma}\label{lem:anchor}
Like in the previous sections we take the masses as $m_1 \becomes \mu$, $m_2 \becomes \mu$ and $m_3 \becomes 1$. In region 1, for example, for $\beta_1 = \beta_2 = 1$, the polynomial $f$ has one real zero in the interval $I_3$. 
\end{lemma}
\begin{proof}
With the masses and parameters as stated the polynomial factorizes as $f(u) = \mu (1 + \mu) (u - 1) (1 + 3u + 5u^2 + 3u^3 + u^4 + \mu (1 + u)^4)$. Then $f$ has a real zero at $u = 1 \in I_3$ and the other zeros of $f$ are complex (in fact two reciprocal complex conjugate pairs).
\end{proof}
Using these two lemmas we can easily construct the numbers of zeros as indicated in figure~\ref{fig:gregions}. Starting in region 1, $f$ has only one real zero which lies in $I_3$. Crossing the $\beta_2$-axis into region 11, the number of zeros in both $I_1$ and $I_3$ increases or decreases by one. Since in region 1 there are no zeros of $f$ in $I_1$, the number of zeros increases and the number of zeros of $f$ for $(I_1,I_2,I_3)$ changes from $(0,0,1)$ to $(0,1,2)$. Starting again in region 1 but now crossing the $\beta_1$-axis into region 12, the number of zeros of $f$ changes from $(0,0,1)$ to $(1,0,2)$. Let us now start in region 12, crossing the curve $\Gamma$ occurs on the image of $I_3$ and the number of zeros on $I_3$ changes by two. However we can not decide whether it is an increase or a decrease. Therefore we start again in region 12 but now cross the $\beta_2$-axis into region 13. Then the number of zeros in both $I_1$ and $I_3$ changes by one. Since in region 12 there are no zeros of $f$ in interval $I_1$, the change must be an increase. Reasoning in this manner we obtain all triples.

\subsubsection{The number of collinear central configurations}\label{sec:numccc}
From the previous sections we know the number of solutions of $f(u) = 0$ in each interval $I_1$, $I_2$ and $I_3$. In appendix~\ref{sec:gactions} we consider the action of the permutation group on the reduced space of collinear configurations and the parameter space. This leads to the following. Fixing the order of the bodies such that the differences $x = r_2 - r_3$ and $y = r_3 - r_1$ have equal sign is equivalent to requiring $u \in I_3$ on the reduced space. This means that only solutions of $f(u) = 0$ in $I_3$ correspond to collinear central configurations. This is the choice we made in section \ref{sec:peqccc}.
\begin{figure}[htbp]
\setlength{\unitlength}{1mm}
\begin{picture}(60,115)(-3,-5)
\put(0,    0){\includegraphics[scale=0.7]{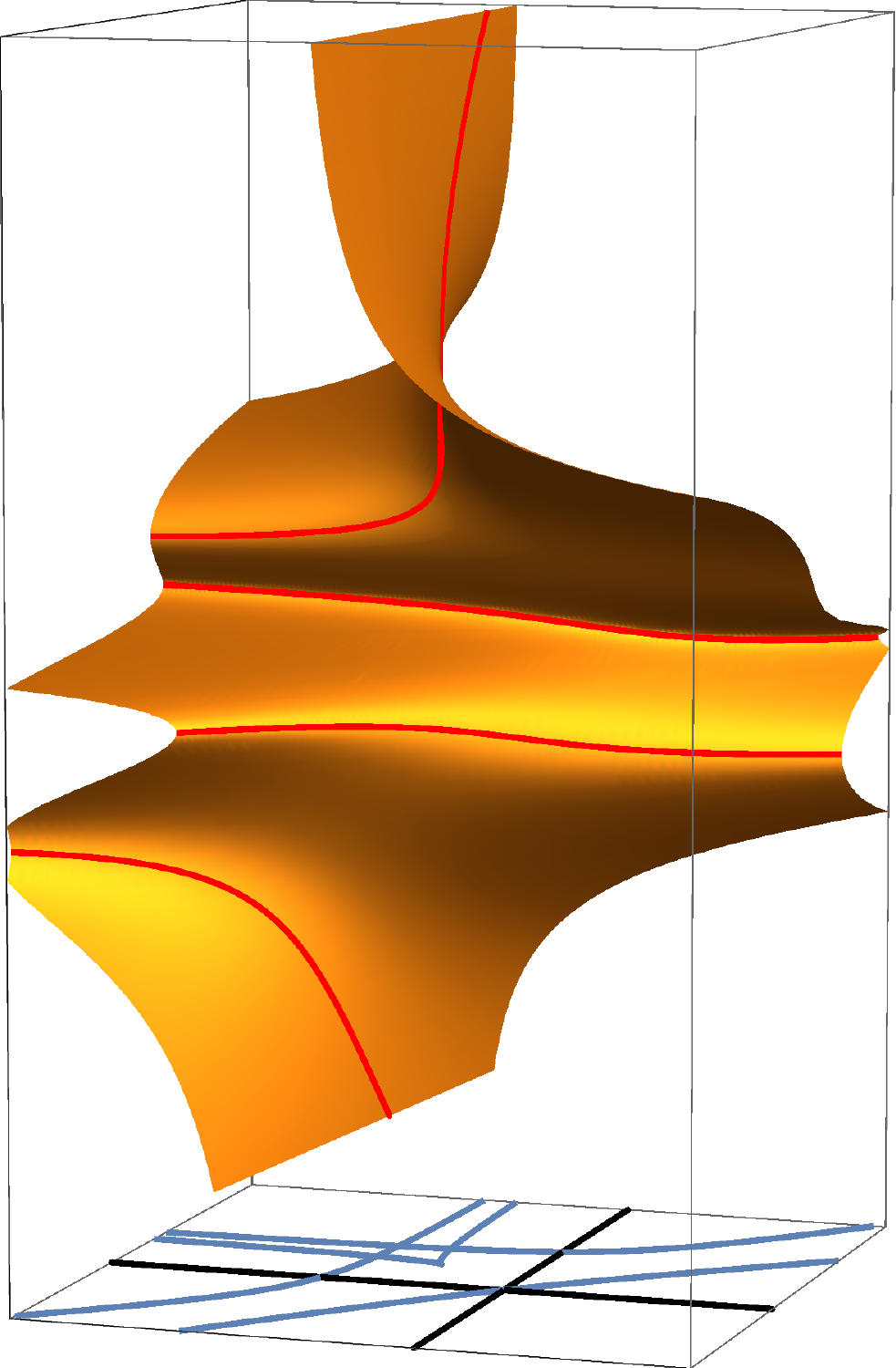}}
\put(30,  -2){$m_1$}
\put(62,   1){$m_2$}
\put(2,   99){$u$}
\end{picture}
\begin{picture}(60,110)(-20,-5)
\put(0,    0){\includegraphics[scale=0.6]{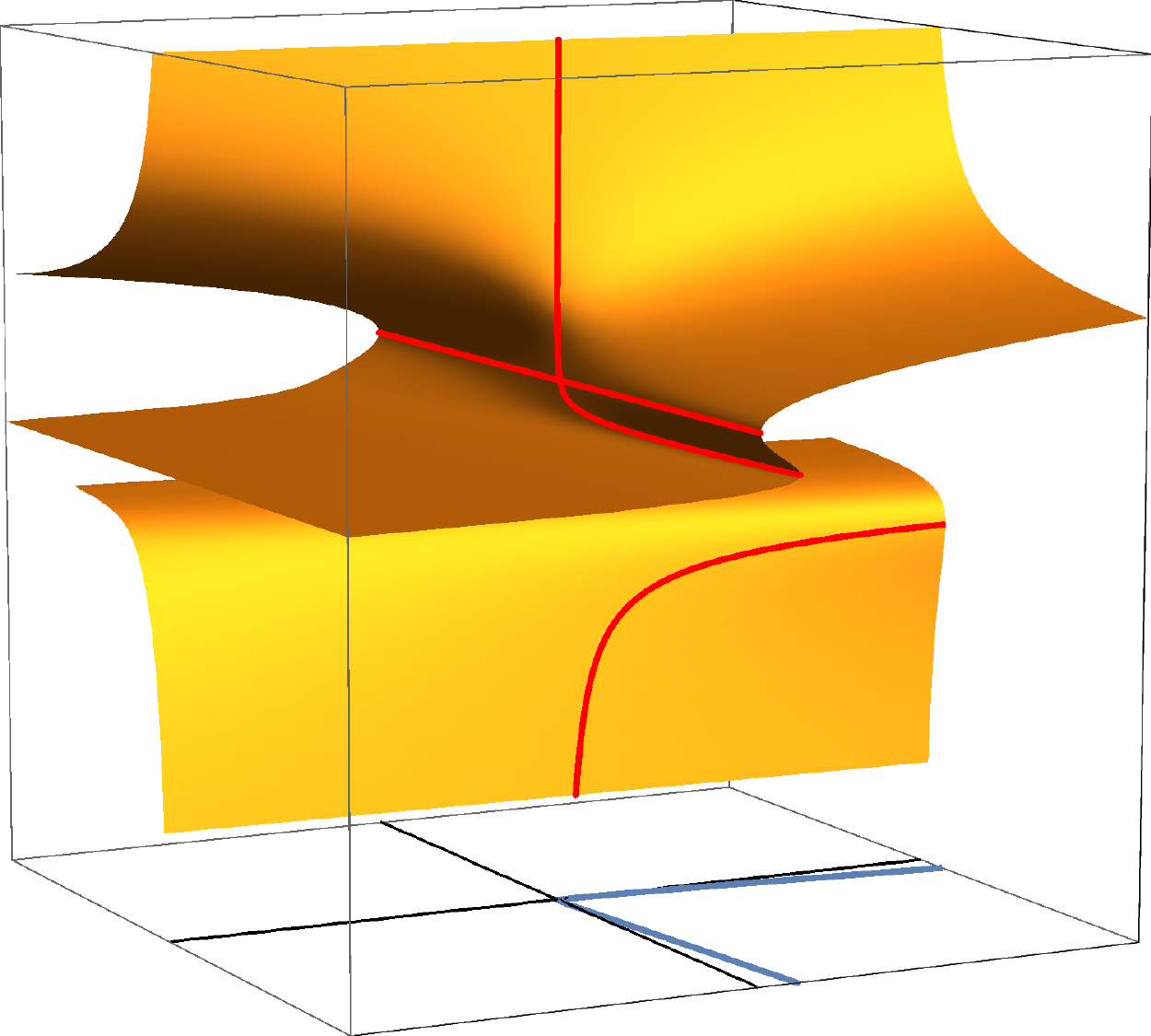}}
\put(51,   0){$\beta_1$}
\put(62,  14){$\beta_2$}
\put(2,   63){$u$}
\end{picture}
\caption{\textit{Left: The surface defined by $f(u,m) = 0$ for a three particle system with gravitational interaction. The fold lines (red) are projected on the $(m_1, m_2)$-parameter plane. To obtain $f(u,m)$ take $\alpha_i= m_j m_k$ with $(i,j,k)=(1,2,3)$ and cyclic permutations in equation~\eqref{eq:fdecomp}. Right: The surface defined by $f(u) = 0$ for a charged three-body system. The fold lines (red) are projected on the $(\beta_1, \beta_2)$-parameter plane. To obtain $f(u)$ take $\alpha_1 = \beta_1$, $\alpha_2 = \beta_2$ and $\alpha_3 = 1$ in 
equation~\eqref{eq:fdecomp}.}\label{fig:surface}}
\end{figure}
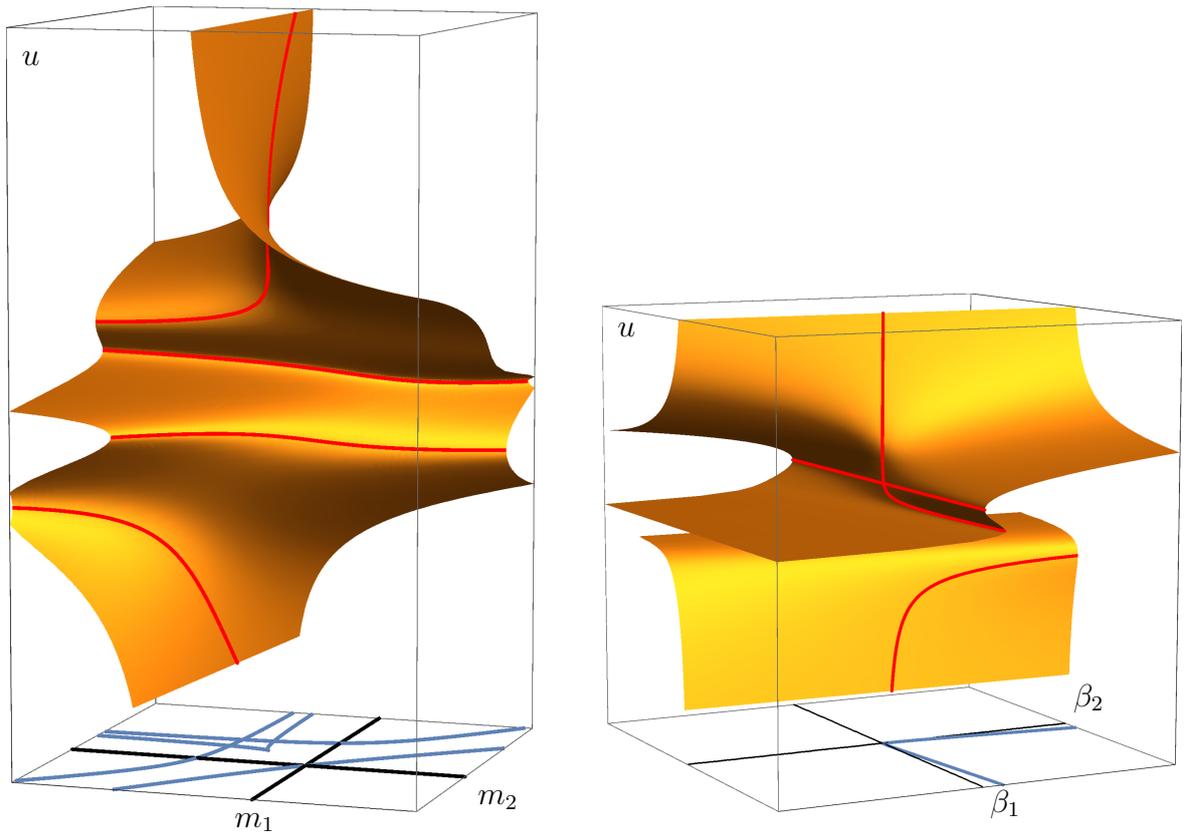

\begin{lemma}\label{lem:i3}
If the bodies are ordered as $(1,3,2)$ or $(2,3,1)$ the collinear central configurations correspond to solutions of $f(u) = 0$ in the interval $I_3$ of the reduced space of collinear configurations. Depending on the values of the parameters there are $0$, $1$, $2$ or $3$ central configurations.
\end{lemma}
This follows immediately from table \ref{tab:gregions}. If we permute the bodies, we know from the action of the permutation group that the central configurations are determined by a similar equation. However by the action of the symmetry group we know how the regions in the parameter plane of this new equation are related to regions of the equation $f(u) = 0$. Thus we have the following corollary.

\begin{corollary}\label{cor:numberccc}
The number of collinear central configurations in a charged three-body system with potential $V$ from equation~\eqref{eq:potential} is up to permutations of the bodies equal to $0$, $1$, $2$ or $3$ depending on the values of the parameters.
\end{corollary}

\commentaar{following to discussion?}
\subsubsection{Miscellaneous remarks}\label{sec:miscrem}

In the following we give several remarks on the previous analysis.

\begin{remark}\label{rem:misc}\strut
\begin{enumerate}[topsep=0ex,partopsep=0px,itemsep=0px]
\item The permutation group consists of six elements. But we see only three different orders in the reduced space of collinear configurations. The reason is that, for example, the permutation $(1,2,3) \mapsto (3,2,1)$ can be regarded as a rotation over $\pi$ in configuration space.
\item Using the same method for the gravitational three-body system we find a single equation for the collinear central configurations on the reduced space of collinear configurations. In fact we find a three-parameter family of equations where the masses of the bodies are the parameters, namely
\begin{align*}
f(u) &= (m_2 + m_3) u^5 + (3m_2 + 2m_3) u^4 + (3m_2 + m_3) u^3\\
     &- (3m_1 + m_3) u^2 - (3m_1 + 2m_3) u - (m_1 + m_2).
\end{align*}
Using the fact that $f$ is homogeneous in $m$ we may divide by, for example, $m_3$  or equivalently set $m_3 = 1$. Thus we consider a two-parameter family. The discriminant set of $f$ is a set of curves in the $(m_1, m_2)$-parameter plane. These can be viewed as fold lines under the projection of the surface defined by $f(u) = 0$ in the $u$-direction, see figure \ref{fig:surface}. In this case however it is not hard to show that the curves constituting the discriminant set stay outside the first quadrant in the $(m_1, m_2)$-parameter plane. Since the masses are positive there is a fixed number of collinear central configurations. To find this number we only need to check one example. Taking $m_1 = m_2 = 1$ we see that there is only one collinear central configuration for $u = 1$.
\end{enumerate}
\end{remark}

\subsection{Non-collinear central configurations}\label{sec:nccc}
To find the non-collinear central configurations we also apply reduction with respect to translation and rotation symmetry. To find the equation for central configurations in reduced coordinates we use the characterization of central configurations as critical points of the potential restricted to levels of the moment of inertia.

\subsubsection{The reduced configuration space}
Ignoring the collision set $\Delta_c$ the configuration space of a three-body system is $\R^9=\R^{3}\times \R^3 \times \R^3 $. Reduction with respect to the group of translations $\R^3$ gives the translation reduced space with topology $\R^6$ which is commonly chosen as the space $\{q\in\R^9\;|\; Q=\frac{1}{m}\sum_i m_i q_i=0\}$, i.e. the subspace of $\R^9$ on which the center of mass $Q$ is vanishing. Instead the reduction with respect to translations can also be used to achieve $q_1=0$. We will take this point of view. For a non-collinear configuration, one can then achieve by rotation about the origin that $q_2$ is located on the positive $x$-axis (the first axis in an Euclidean space $\R^3$ which has the origin at $q_1$). By another rotation about the $x$-axis one can achieve that $q_3$ is contained in the upper half of the $(x,y)$-plane (with $y$ corresponding to the second axis of the Euclidean space). The configuration space reduced by translations and rotations hence has the topology $\R^3_+$. As coordinates on $\R^3_+$ we can take the mutual body distances $r_{ij} = \norm{q_i-q_j}$ which are non-negative and need to satisfy the triangle inequality. As the reduced configuration space we hence get $\{(r_{12}, r_{23}, r_{13}) \in \R_+^3\;|\; r_{ij} + r_{jk} \geq r_{ki} \;\text{for mutually different}\;i,j,k \in \{1,2,3\} \}$. The latter is a cone over a triangle, containing three singular planes which intersect in three singular lines which meet in a singular point at the vertex.\commentaar{komt overeen met singuliere $\SO(3)$ banen, op in gaan?} 

\subsubsection{Non-collinear central configurations on the reduced space}
In order to obtain an equation for central configurations in the reduced configuration space it is useful to note that central configurations can be obtained as critical points of the potential $V$ restricted to the levels of constant moment of inertia. The (polar) moment of inertia is defined as 
\begin{equation*}
I(q) = \sum_i m_i \norm{q_i-Q}^2.
\end{equation*}
Then we have the following.
\begin{lemma}\label{lem:ccvi}
A central configuration is a critical point of $V$ restricted to the levels of $I$.
\end{lemma}
\begin{proof}
The proof follows immediately from the Lagrange multiplier theorem.
\end{proof}
This characterization is useful on the reduced space on which both $V$ and $I$ can be expressed as functions of the $\SE(3)$ invariants $r_{ij}$. We already know $V$ on the reduced space. 
For $I$, we have the following.\commentaar{met een bewijs(je)?}
\begin{lemma}\label{lem:invari}
On the reduced space $$I(r_{12},r_{13},r_{23}) = \frac{1}{2m}( m_1 m_2 r^2_{12} + m_1 m_3 r^2_{13} + m_2 m_3 r^2_{23} ),$$ where $m = m_1+m_2+m_3$.
\end{lemma} 
Using lemma \ref{lem:ccvi} we readily obtain the non-collinear central configurations from 
$\big( \partial_{r_{12}} V, \partial_{r_{13}} V, \partial_{r_{23}} V \big)$ and $\big( \partial_{r_{12}} I, \partial_{r_{13}} I, \partial_{r_{23}} I \big)$ being parallel
as 
$r_{ij} = \big(\frac{\gamma_{ij} M}{m_i m_j \lambda}\big)^{\frac{1}{3}}$ where $\lambda$ is the Lagrange multiplier provided that the $r_{ij}$ satisfy the triangle inequalities. In the gravitational case we have $\gamma_{ij} = m_i m_j$ which give equal $r_{ij} = \big( \frac{m}{\lambda} \big)^{\frac{1}{3}}$. This is the equilateral triangle solution named after Lagrange. In the case of electrostatic interaction we have $r_{ij} = \big(-\frac{m}{\lambda} \frac{ Q_i Q_j}{m_i m_j }\big)^{\frac{1}{3}}$ where the $Q_i$ are the charges. Only when the charges have equal sign, a non-linear central configuration may exist. In the case of equal signs the existence of a non-linear central configuration can still be obstructed since the triangle inequalities cannot be satisfied for the given ratios of masses and charges $Q_i/m_i$, $i=1,2,3$. In summary we have the following theorem.
\begin{theorem}\label{the:nccc}
A charged three-body system with a potential as in equation~\eqref{eq:potential} has at most one non-collinear central configuration. A necessary (but not sufficient) condition for the existence of a non-collinear central configuration is that the $\alpha_{i}$ in equation~\eqref{eq:potential} have the same sign.
\end{theorem}

\section{Critical points in the charged $3$-body system}\label{sec:criptsctbp}
\commentaar{
onderwerpen
- Critical points related to collinear central configurations
- Critical points related to non-collinear central configurations
- Other critical points?
- Morse type of critical points?
volgorde
- alleen kritieke punten bij centrale configs (ccs)
- alleen ccs met positieve mult komen in aanmerking
- gebruik dat de pot homogeen van de graad -1 is: teken van mult is tegengesteld teken van pot
- onderscheid collin en non-collin ccs
- details over gebieden naar de appendix
}
Let us now consider critical points of the charged $3$-body system related to central configurations. Since every central configuration in a $3$-body system is planar we know from corollary \ref{cor:pccfgtocp} that to each central configuration with a positive multiplier we can associate a critical point. We already determined the central configurations in the previous section, but we still have to determine the signs of the multipliers. To this end we use the following lemma.
\begin{lemma}\label{lem:sign_potential_Lagrange}
For a homogeneous potential of degree $-1$ and a central configuration $q$ with multiplier $\lambda$ it holds that 
$$V(q) = - \lambda I(q),$$
where $I(q)$ is the (polar) moment of inertia around the center of mass $Q$. 
\end{lemma}
\begin{proof}
The equation is obtained from scalar multiplication of the equation defining a central configuration (see definition~\ref{def:ccfg}) with $q_i-Q$, the summation over $i$ and using Euler's homogeneous function theorem.
\end{proof}

Thus the sign of the multiplier can be read off from the sign of the potential. 

\begin{proposition}\label{pro:ctbpcriptsccfgs}
In a $3$-body system with electrostatic interaction there are zero, one, two or three critical points associated to collinear central configurations, depending on the charges. 
There are no critical points associated to non-collinear central configurations.
\end{proposition}

\begin{proof}
For critical points associated to collinear central configurations, we refer to appendix~\ref{sec:signv} on the sign of the potential. For critical points associated to non-collinear central configurations we note that the relative distances of the bodies are given by $r_{ij} = \big(-\frac{m}{\lambda} \frac{ Q_i Q_j}{m_i m_j }\big)^{\frac{1}{3}}$ where the $Q_i$ are the charges. This implies that the charges must have equal signs and hat the multiplier $\lambda$ must be negative. Now it follows from corollary \ref{cor:cptopccfg} that there can be no critical point associated to a non-collinear central configuration.
\end{proof}

\section{Remarks and outlook}\label{sec:outlook}
\commentaar{
- integral manifolds: critical points and central configurations
- Newton versus Coulomb
- dynamics: critical points and relative equilibria, stability
- ...
}

\begin{figure}[htbp]
\begin{center}
\raisebox{3cm}{a)}\includegraphics[height=3.5cm]{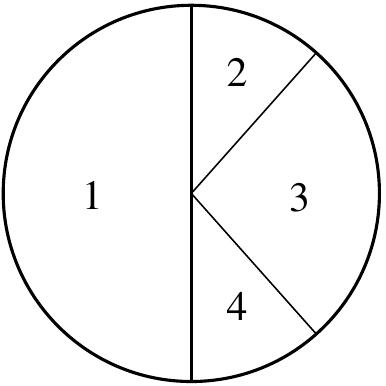}
\hspace*{1cm}
\raisebox{3cm}{b)}\includegraphics[height=3.5cm]{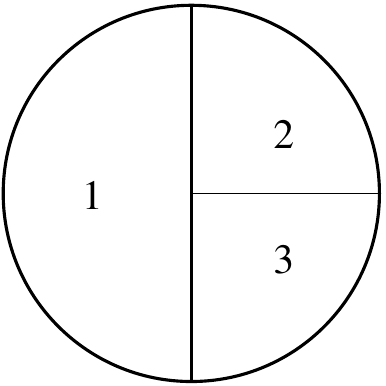}
\hspace*{1cm}
\raisebox{3cm}{c)}\includegraphics[height=3.5cm]{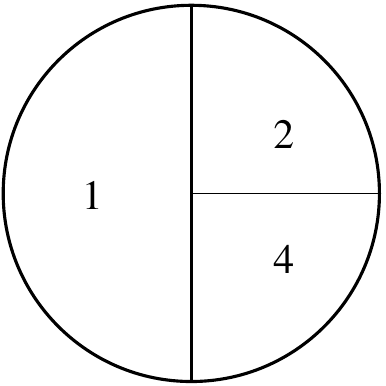}
\end{center}
\caption{\textit{Schematic plot illustrating the relation of critical points of the integral map to central configurations  (see main text) for a) 
a (general) charged $N$-body system, b) a gravitational $N$-body system, and c) a  (general) three-body system with electrostatic interactions.}\label{fig:cc_vs_relequilibria}}
\end{figure}

In this paper we studied the charged three-body problem. 
Finding solutions of the equations of motion is a formidable task. 
As solutions are confined to the joint level sets of the integrals (constants of motion) some information on the solutions can be obtained from the geometry of the level sets, or in other words, the integral manifolds which are the fibres of the integral map. In the gravitational case the topology of the fibres has been studied by McCord, Meyer and Wang  \cite{mmw1998}. The present paper comprises the 
first steps towards a generalisation of their results to more general potentials. The 
topology of the integral manifolds may change at a critical value of the integral map. 
In the present article we studied the critical values for charged $N$-body systems and in particular for the case $N=3$. 
The gravitational  $N$-body system is contained in our definition of  a charged $N$-body system as a special case (see definition~\ref{def:NBodySystem}).
In the gravitational case it is known that critical points given by relative equilibria always project to planar central configurations \cite{mmw1998,lms2015}. In this paper we have seen that 
 the situation is more subtle in the more general, charged case.

In order to summarise the results of this paper we consider the schematic plot in figure~\ref{fig:cc_vs_relequilibria}. 
The disks represent the set of critical points of the integral map. 
In propositions~\ref{pro:rankdf_a} and \ref{pro:rankdf} we have seen that the critical points are of  two different types.

The left halves (region 1) of the disks contain critical points which on the space of vanishing center of mass and total momentum are
given by \emph{collinear phase space points} where the position vectors and momentum vectors of all bodies point in the same direction (assuming the center of mass to be at the origin).
These 
 are critical points of the angular momentum and independent of the particular Hamiltonian of the $N$-body system (where we define an $N$-body system according to definition~\ref{def:NBodySystem}).  
The right halves of the disks represent critical points which are relative equilibria, i.e. critical points of the Hamiltonian restricted to the level set of constant angular momentum.
The relative equilibria result from choosing suitable momenta for collinear (region 2) or non-collinear planar (region 3) central configurations, both with positive multiplier, or equivalently negative potential energy (see lemma~\ref{lem:sign_potential_Lagrange}). Here the momenta are chosen in such a way that they give rise to a rigid rotation of the corresponding central configuration. 
Region~4 represent relative equilibria that do not project to central configurations.

For the gravitational $N$-body system, region~4 has disappeared, see figure~\ref{fig:cc_vs_relequilibria}b. In the gravitational $N$-body system all relative equilibria project to 
planar central configurations, and conversely every planar central configuration can be turned into a relative equilibrium by a proper choice of momenta. For the gravitational three-body problem,
region 2 consists of the rigid rotations of collinear central configurations already found by  Euler and region 3 consists of the rigid rotation of an equilateral triangle found by Lagrange.

A three-body system interacting via electrostatic forces has no relative equilibria resulting from non-collinear planar central configurations, i.e. compared to the general case in 
figure~\ref{fig:cc_vs_relequilibria}a region 3 has disappeared, see  figure~\ref{fig:cc_vs_relequilibria}c. This is the contents of corollary~\ref{pro:ctbpcriptsccfgs} which also states that there are (up to similarity transformations) at most three relative equilibria associated with collinear central configurations. 
Moreover, the proof of proposition~\ref{pro:rankdf} shows that there might be relative equilibria that do not project to central configurations (we will give an example in \cite{hwz2018c}).

For a discussion of the stability of the relative equilibria of the charged three-body problem we refer to \cite{apc2008}. Instead our focus in a forthcoming paper will be bifurcations of the non-compact integral manifolds due to critical points of the integral map at infinity \cite{hwz2018b}. For the gravitational case, such critical points have been studied by  
Albouy~\cite{alb1993}. Specific examples and the manifestation of the bifurcations of the integral manifolds in terms of their configuration space projections, the so-called Hill regions, will be given in \cite{hwz2018c}.

\appendix

\section{Special points of the curve $\Gamma$}\label{sec:spoc}
On the curve $\Gamma$ the polynomial $f$ has double zeros. Therefore $\Gamma$ is a fold curve and we may expect to find cusps. Moreover the parameterization of $\Gamma$ is by rational functions. So we may also expect points at infinity on $\Gamma$. For the next exposition, we write the parameterization explicitly (cf equation~\eqref{eq:gammapar})
\begin{equation}\label{eq:dzcurve}
c(u) = \big(\frac{m_3}{m_1} \frac{g_1(u)}{g_3(u)}, \frac{m_3}{m_2} \frac{g_2(u)}{g_3(u)} \big)\;\text{for $u \in \R$},
\end{equation}
where
\begin{align*}
g_1(u) &\becomes m_1 (3 m_1 + m_2 + m_3) u^2 - m_1 (-3 m_1 + m_2 - 3 m_3)u - 2 m_2 (m_1 + m_3)\\
g_2(u) &\becomes u^3 \big[-2 m_1 (m_2 + m_3) u^2 + m_2 (-m_1 + 3 (m_2 + m_3))u\\&+ m_2 (m_1 + 3 m_2 + m_3)\big]\\
g_3(u) &\becomes (1+u)^3 \big[2 m_1 (m_2 + m_3) u^2 + (3 m_3 (m_2 + m_3) + m_1 (4 m_2 + 3 m_3))u\\& + 2 m_2 (m_1 + m_3)\big].
\end{align*}
\begin{figure}[htbp]
\setlength{\unitlength}{1mm}
\begin{picture}(35,60)(0,0)
\put(0,  -7){\includegraphics[scale=0.75]{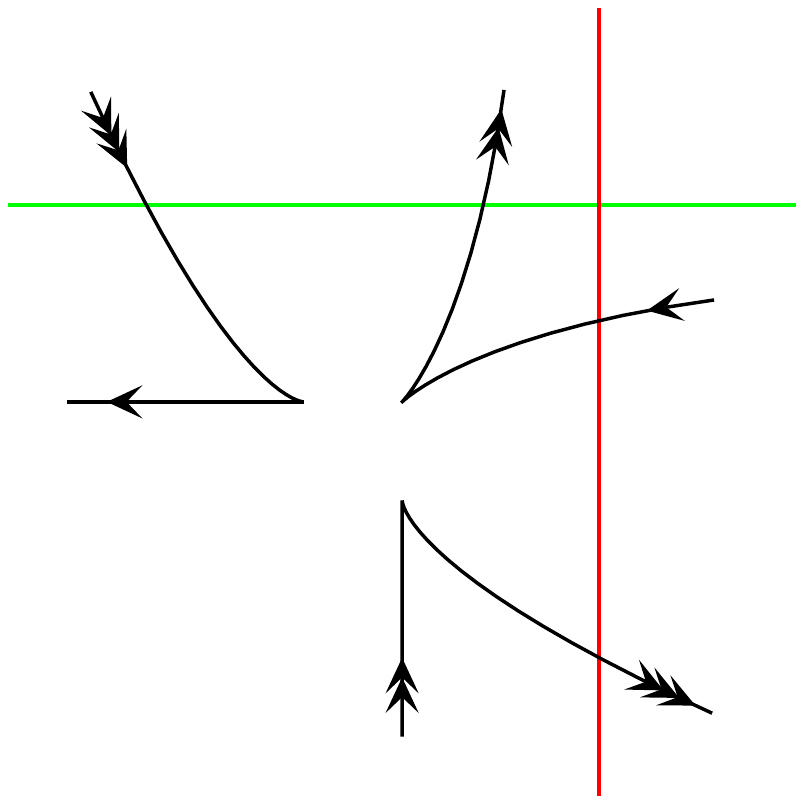}}
\put(69, 45){$\beta_1$}
\put(52, 62){$\beta_2$}
\put(42, 55){\scriptsize{$u=\xi_-$}}
\put(34,  2){\scriptsize{$u=\xi_-$}}
\put(1,  30){\scriptsize{$u=\xi_+$}}
\put(62, 38){\scriptsize{$u=\xi_+$}}
\put(10, 55){\scriptsize{$u=-1$}}
\put(62,  6){\scriptsize{$u=-1$}}
\put(41, 29){\scriptsize{$u=1$}}
\put(41, 22){\scriptsize{$u=\eta_-$}}
\put(26, 27){\scriptsize{$u=\eta_+$}}
\end{picture}
\begin{picture}(80,60)(-33,0)
\put(0,   0){\includegraphics[scale=1.0]{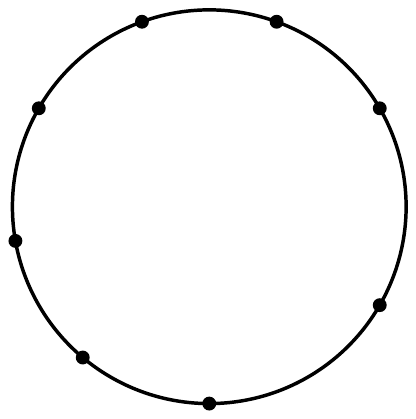}}
\put(6,  40){$-1$}
\put(49, 40){$0$}
\put(49, 17){$1$}
\put(28,  7){$\infty$}
\put(18, 16){$\xi_-$}
\put(12, 27){$\eta_-$}
\put(14, 38){$\xi_0$}
\put(23, 44){$\eta_+$}
\put(34, 44){$\xi_+$}
\put(42, 21){$\eta_0$}
\end{picture}
\caption[Schematic picture of the curve $\Gamma$]{\textit{Left: Schematic picture of the curve $\Gamma$. The arrows indicate how the three branches are connected at infinity. For $u \in \{\xi_-, \xi_0, \xi_+\}$, $\Gamma$ has points at infinity, for $u \in \{\eta_-, \eta_0, \eta_+\}$, $\Gamma$ has cusp points. Right: Special values of $u$ on the reduced space.}\label{fig:gcurve}}
\end{figure}
To keep the computations tractable we now set $m_1 \becomes \mu$, $m_2 \becomes \mu$ and $m_3 \becomes 1$. The consequences of this choice are:

\begin{enumerate}
\item When $m_1 = m_2 = \mu$ the curve $\Gamma$ has a symmetry property. Let $T(\beta_1, \beta_2) = (\beta_2, \beta_1)$ then $T c(u) = c(\frac{1}{u})$. This will be explained in more detail in 
appendix~\ref{sec:gactions}.
\item A consequence of the previous is that if $\eta$ is a value such that $c(\eta)$ is a singular point, then $c(\eta^{-1})$ is also a singular point. If $\Gamma$ has an odd number of singular points, one must occur for $u = 1$ or $u = -1$. The singular points of $\Gamma$ are found by solving $c'(u) = 0$.
\item Similarly, if $\xi$ is a value such that $c(\xi)$ is a point at infinity (taking an appropriate limit), then $c(\xi^{-1})$ is also a point at infinity. In case $\Gamma$ has an odd number of such points, one must occur for $u = 1$ or $u = -1$.
\end{enumerate}
With these choices for the masses $m_i$ we have the following result.
\begin{lemma}\label{lem:dzcurve}
The curve $\Gamma$ defined in equation~\eqref{eq:dzcurve} has six special points.
\begin{enumerate}[label=(\roman*)]
\item There are three points at infinity defined by the zeros of $g_3$ (see~\eqref{eq:dzcurve}) namely at $u$ equal to 
\begin{equation}
\begin{aligned}
\xi_{\pm} &= \frac{-3 - 6 \mu - 4 \mu^2 \pm \sqrt{9 + 36 \mu + 44 \mu^2 + 16 \mu^3}}{4\mu(1+\mu)} \text{ or }\\
\xi_0 &= -1.
\end{aligned}
\end{equation}
For $\xi_{\pm}$, the corresponding points on $\Gamma$ are regular points at infinity. Note that $\xi_- \xi_+ = 1$. For $\xi_0$ the corresponding point is a degenerate cusp at infinity, with equivalent local parameterization $t \mapsto (t^3 + \bigo(t^4), t^4 + \bigo(t^5))$. The curve $\Gamma$ has three branches, namely $c\big((\xi_-, \xi_0)\big)$, $c((\xi_0, \xi_+\big)$ and $c\big((\xi_+,\infty) \cup (\infty, \xi_-\big)$, see figure \ref{fig:gcurve}.
\item There are three singular points at $u$ taking one of the values
\begin{equation}
\begin{aligned}
\eta_{\pm} &= \frac{-5 - 12 \mu - 8 \mu^2 \pm \sqrt{21 + 80 \mu + 92 \mu^2 + 32 \mu^3}}{2(1 + 5\mu + 4 \mu^2)}\text{ or } \\
\eta_0 &= 1.\\
\end{aligned}
\end{equation}
Each singular point is a regular cusp with an equivalent local parameterization $t \mapsto (t^2 + \bigo(t^3), t^3 + \bigo(t^4))$. Note that $\eta_- \eta_+ = 1$.
\end{enumerate}
The inequalities $\xi_- < \eta_- < \xi_0 < \eta_+ < \xi_+ < \eta_0$, valid for all $\mu > 0$, imply that the six special points on $\Gamma$ are alternately cusps and points at infinity.
\end{lemma}

\begin{proof}
Points at infinity are found from solving $g_3(u) = 0$. Since $g_3$ has an easy factorization, see equation \eqref{eq:dzcurve}, we immediately find the solutions presented in the lemma. For each solution, we have $g_1(\xi, m) \neq 0$ and $g_2(\xi, m) \neq 0$ so that we find indeed points at infinity. The singular points are found from solving $c'(\eta) = (0, 0)$. After some computations we find that $\eta$ must be a zero of the polynomial
\begin{equation*}
f_3f_2'f_1'' - f_2f_3'f_1'' - f_3f_1'f_2'' + f_1f_3'f_2'' - f_2f_1'f_3'' - f_1f_2'f_3'',
\end{equation*}
where the $f_i$ are defined in equation \eqref{eq:fdecomp}. For our choice $m_1 = m_2 = \mu$ and $m_3 = 1$, this polynomial factorizes as 
\begin{equation*}
(u-1) u^2 (u+1)^2 \big(1 + 5 \mu + 4 \mu^2 + 5u + 12 \mu u + 8 \mu^2 u + u^2 + 5 \mu u^2 + 4 \mu^2 u^2\big)
\end{equation*}
from which we readily find the values of $\eta$. For a local study of the singular points, we look at the Taylor series of $A \big[c(\eta+t) - c(\eta) \big]$, where $A$ is a linear transformation that locally puts $\Gamma$ in a standard form, namely $A = \mat(Jv,v)$ where $v$ is the tangent vector of $\Gamma$ at the singularity, $v \becomes \frac{d}{dt}(c(\eta+t) - c(\eta))|_{t=0}$, and $J$ is the matrix of a rotation over $\tfrac{\pi}{2}$. For each value of $\eta$, we find
\begin{equation*}
A \big[c(\eta+t) - c(\eta) \big] = \big(\gamma_1(\mu) t^3 + \bigo(t^4), \gamma_2(\mu) t^2 + \bigo(t^3)\big)
\end{equation*}
where the $\gamma_i$ are functions of $\mu$. In general they are non-zero for positive values of $\mu$. Their explicit expressions are quite involved except for $\eta = 1$. Then we have $\gamma_1(\mu) = \frac{7+8\mu}{24+80\mu+64\mu^2}$ and $\gamma_2(\mu) = \frac{3(7+8\mu)}{8(3+4\mu)^2}$. For a local study of points at infinity, we look at the Taylor series of $\tfrac{c}{\norm{c}^2}$. Let $\xi$ be a value such that $c(\xi)$ ``is'' a point at infinity. Then we find
\begin{equation*}
\frac{c(\xi+t)}{\norm{c(\xi+t)}^2} = \big(\gamma_1(\mu) t + \bigo(t^2), \gamma_2(\mu) t + \bigo(t^2) \big)
\end{equation*}
for $\xi = \xi_-$ and $\xi = \xi_+$. So these are regular points since the $\gamma_i$ are positive for $m$ positive. For $\xi = -1$ however, we find after a linear transformation
\begin{equation*}
A\frac{c(\xi+t)}{\norm{c(\xi+t)}^2} = \big(\frac{3+2\mu}{4} t^3 + \bigo(t^4), \frac{3}{16}(3+2\mu)^2 t^4 + \bigo(t^5) \big).
\end{equation*}
Even for general values of $m_1$, $m_2$ and $m_3$ we find
\begin{equation*}
A(m_1,m_2,m_3)\frac{c(\xi+t)}{\norm{c(\xi+t)}^2} = \big(a_1(m_1,m_2,m_3) t^3 + \bigo(t^4), a_2(m_1,m_2,m_3) t^4 + \bigo(t^5) \big),
\end{equation*}
where the $a_i$ are positive. Therefore at $\xi = -1$ we have a singular point at infinity.
\end{proof}
We present a schematic picture of the curve $\Gamma$ in figure \ref{fig:gcurve}. In particular the curve is not to scale near the cusps. This picture however serves to show the global structure and the alternation of singular points and points at infinity. Moreover from this picture we get regions in the parameter plane with a constant number of real solutions of $f$.

\section{The action of the permutation group}\label{sec:gactions}
The main purpose of this section is corollary \ref{cor:cccfgs} which states that all central configurations can be found from equation $f(u) = 0$ in proposition \ref{pro:redeq}. To reach this conclusion we have to consider the action of the permutation group on the reduced space of collinear configurations. The permutation group $S_3$ of three bodies is generated by $\pi_1 = (1\;2)$ and $\pi_2 = (2\;3)$. Then the other elements are $\pi_0 = \id$, $\pi_3 = \pi_1 \circ \pi_2$, $\pi_4 = \pi_2 \circ \pi_1$ and $\pi_5 = \pi_1 \circ \pi_2 \circ \pi_1$. Consider the product space of collinear configurations and the space of parameters $(r, \alpha, m) \in \R^3 \times \R^3 \times \R^3$. Let $S_3$ act on each factor in the same standard way, then we immediately have the following.
\begin{lemma}\label{lem:invperm}
The equation for (collinear) central configurations in definition~\ref{def:ccfg} with potential $V$ from equation~\eqref{eq:potential} is invariant with respect to the action of $S_3$ defined above.
\end{lemma}
Now we consider the reduced space of collinear configurations. Then there is an induced $S_3$-action on this space. Using the intermediate variables $x = r_2 - r_3$, $y = r_3 - r_1$ and $z = r_1 - r_2$ and the coordinate $u$ on the reduced space, defined by $y = ux$ we readily obtain the following result.
\begin{lemma}\label{lem:indact}
The induced $S_3$-action on the reduced space of collinear configurations is determined by the homomorphisms $h_i$ with: $h_0 = \id$, $h_1(u) = \frac{1}{u}$, $h_2(u) = -(1+u)$, $h_3 = h_1 \circ h_2$, $h_4 = h_2 \circ h_1$ and $h_5 = h_1 \circ h_2 \circ h_1$.
\end{lemma}
Note that the three disjoint intervals $I_1$, $I_2$ and $I_3$ in $\R \cup \{\infty\}$, see \ref{def:uintervals}, are permuted by the $h_i$, see table \ref{tab:iperm}.

\begin{table}[htbp]
\begin{center}
\begin{tabular}{l|ccc}
      & $I_1$ & $I_2$ & $I_3$\\\hline
$h_1$ & $I_2$ & $I_1$ & $I_3$\\
$h_2$ & $I_3$ & $I_2$ & $I_1$
\end{tabular}
\end{center}
\caption{\textit{Permutation of the intervals $I_1$, $I_2$ and $I_3$ by $h_1$ and $h_2$.}\label{tab:iperm}}
\end{table}

On the reduced space of collinear configurations equation $f(u) = 0$, see \eqref{eq:fdecomp}, determines the collinear central configurations. The polynomial $f$ depends on parameters $m$ and $\alpha$ and we explicitly write $f(u; \alpha, m)$ to stress that fact. In deriving $f$ we made a choice for the signs of the intermediate variables $x$, $y$ and $z$. Permuting the bodies will also change the signs of the transformed $x$, $y$ and $z$. From the very beginning we could have absorbed these signs into the parameters $\alpha_i$. This would give the equations \eqref{eq:intermed} a more symmetric appearance ($\alpha_3 \mapsto -\alpha_3$), but obfuscate the sign dependence. Therefore we chose fixed signs, but this means that the $S_3$-action on $\alpha$ is non-standard on the reduced space. In fact we have $\phi_0 = id$, $\phi_1(\alpha) = (\alpha_2, \alpha_1, \alpha_3)$, $\phi_2(\alpha) = (\alpha_1, -\alpha_3, -\alpha_2)$, $\phi_3 = \phi_1 \circ \phi_2$, $\phi_4 = \phi_2 \circ \phi_1$ and $\phi_5 = \phi_1 \circ \phi_2 \circ \phi_1$. A short computation immediately reveals the following.
\begin{lemma}\label{lem:fperm}
The polynomial $f$ transforms under the generators of the $S_3$ action on $(r, \alpha, m)$ according to
\begin{align*}
-u^5 f\big(h_1(u); \phi_1(\alpha), \pi_1(m)\big) &= f(u; \alpha, m),\\
    -f\big(h_2(u); \phi_2(\alpha), \pi_2(m)\big) &= f(u; \alpha, m)
\end{align*}
\end{lemma}
Thus we have the following.
\begin{corollary}\label{cor:cccfgs}
The collinear central configurations for permuted bodies follow the equation $f(u) = 0$ in proposition \ref{pro:redeq} with the particular sign choice, using the transformations from lemma \ref{lem:fperm}.
\end{corollary}

Before finding the discriminant set of $f$ we simplified the problem by using the fact that $f$ is homogeneous of degree 1 in $\alpha$ and assuming $\alpha_3 \neq 0$ we scaled $(\alpha_1,\alpha_2,\alpha_3) \mapsto \tfrac{1}{\alpha_3}(\alpha_1,\alpha_2,\alpha_3)$. Then we get an $S_3$-action $\psi$ on the two-parameter space with coordinates $(\beta_1, \beta_2) = \beta(\alpha) = (\frac{\alpha_1}{\alpha_3}, \frac{\alpha_2}{\alpha_3})$ defined by $\psi_i(\beta(\alpha)) = \beta(\phi_i(\alpha))$. In particular we have
\begin{align*}
\psi_1(\beta_1, \beta_2) &= (\beta_2, \beta_1),\\
\psi_2(\beta_1, \beta_2) &= (-\frac{\beta_1}{\beta_2}, \frac{1}{\beta_2}).
\end{align*}
One of the components of the discriminant set of $f$ is the curve $\Gamma$. Using the lemmas and definitions above we arrive at the final result.
\begin{corollary}\label{cor:gperm}
Under permutations of the bodies, the curve $\Gamma$ transforms via its parameterization $c$ as
\begin{equation*}
\psi_i\Big(c\big(h_i(u); \pi_i(m)\big)\Big) = c(u; m)
\end{equation*}
for $i \in \{0,1,2,3,4,5\}$.
\end{corollary}

\section{The sign of the potential}\label{sec:signv}
In this section we only consider \emph{collinear} central configurations. For brevity we drop \emph{collinear}. We wish to determine the sign of the potential $V$ at a central configuration. The latter are determined by a fifth degree polynomial $f$, see equation \eqref{eq:fdecomp}, and thus can not be found explicitly. We can however find central configurations at which the potential is zero because the polynomial $f$ depends on parameters. In fact this way we find regions in the $(\beta_1, \beta_2)$-parameter plane, see appendix~\ref{sec:spoc}, where the numbers of central configurations with positive and negative potential values do not change.

To this end we eliminate $u$ from the equations
\begin{equation*}
\system{
f(u) &= 0,\\
U(u) &= 0
}
\end{equation*}
where $U$ is the potential on the reduced space of collinear configurations. For sake of completeness note that
\begin{equation*}
U(u) = -\frac{\beta_1 u^2 + (1 + \beta_1 + \beta_2) u + \beta_2}{u(1+u)}.
\end{equation*}
After eliminating $u$ we find a relation for $\beta$ defining a curve bounding the aforementioned regions
\begin{equation*}
(\beta_1 - \beta_2)^2 + 2(\beta_1 + \beta_2) + 1 = 0.
\end{equation*}
Note that this equation does not involve the masses. The curve defined by this equation is a parabola and it is not hard to see that the curve $\Gamma$ lies on one side of it. Thus the regions of constant numbers of zeros of $f$ in the intervals $I_1$, $I_2$ and $I_3$ are not subdivided any further apart from region 6.

Now we count the number of central configurations with negative value of the potential in the interval $I_3 = (0, \infty)$ only since the others follow by the action of the permutations, see appendix~\ref{sec:gactions}. Taking any point from the various regions we find the following table.

\begin{table}[htbp]
\begin{center}
\begin{tabular}{l|c|c}
region & zeros     & signs of $V$ \\\hline
1      & $(0,0,1)$ & $(,,-)$\\
2      & $(0,2,1)$ & $(,--,-)$\\
3      & $(0,1,0)$ & $(,-,)$\\
6      & $(1,1,1)$ & $(+,-,+)$ or $(+,-,-)$\\
9      & $(1,0,0)$ & $(+,,)$\\
10     & $(2,0,1)$ & $(++,,-)$\\
11     & $(1,0,2)$ & $(+,,--)$\\
12     & $(0,1,2)$ & $(,-,--)$\\
13     & $(1,1,3)$ & $(+,-,---)$
\end{tabular}
\end{center}
\caption[The number of zeros of $f$]{\textit{The number of zeros of $f$ and the signs of the potential in intervals $I_1$, $I_2$ and $I_3$.}\label{tab:vregions}}
\end{table}


\end{document}